\documentclass[11pt]{amsart}

\usepackage{epigamath}

\usepackage[english]{babel}

\numberwithin{equation}{section}

\usepackage[shortlabels]{enumitem}

\usepackage{subfiles}

\usepackage{graphicx}			
\usepackage{amssymb,amsfonts}
\usepackage{mathrsfs}
\usepackage{amsthm}
\usepackage{amsmath}
\usepackage{tikz}
\usepackage{tikz-cd}
\usepackage{accents}
\usepackage{upgreek}
\usepackage{mathtools}
\usepackage{caption}
\usepackage{url}
\usepackage{float}
\usepackage{todonotes, derivative}
\usepackage{bbm} 

\usepackage{xparse}
\usepackage{xfrac}

\tikzset{
  commutative diagrams/.cd, 
  arrow style=tikz, 
  diagrams={>=stealth}
}

\usetikzlibrary{calc}
\usetikzlibrary{fadings}
\usetikzlibrary{decorations.pathmorphing}
\usetikzlibrary{decorations.pathreplacing}

\theoremstyle{plain}
\newtheorem{theorem}{Theorem}[section]
\newtheorem{corollary}[theorem]{Corollary}
\newtheorem{lemma}[theorem]{Lemma}
\newtheorem{proposition}[theorem]{Proposition}

\newtheorem{maintheorem}{Theorem}
\newtheorem{maincorollary}[maintheorem]{Corollary}

\theoremstyle{remark}
\newtheorem{remark}[theorem]{Remark}

\newtheorem{setting}[theorem]{Setting}

\theoremstyle{definition}
\newtheorem{definition}[theorem]{Definition}

\DeclareMathOperator{\ev}{\mathsf{ev}}
\DeclareMathOperator{\Hom}{Hom}

\newcommand{\cY}{\mathcal{Y}}

\newcommand\fd{\mathfrak d}
\newcommand\fD{\mathfrak D}
\newcommand\fm{\mathfrak m}

\newcommand\bk{\mathbf{k}}

\newcommand{\CC}{\mathbb{C}}
\newcommand{\NN}{\mathbb{N}}
\newcommand{\PP}{\mathbb{P}}
\newcommand{\RR}{\mathbb{R}}
\newcommand{\ZZ}{\mathbb{Z}}
\newcommand{\AAA}{\mathbb{A}}

\newcommand{\vir}{\text{\rm vir}}
\newcommand{\Ntrop}{{N}_{\mathsf{trop}}^{\Delta,\bk}}

\newcommand{\pt}{{\mathsf{pt}}} 

\DeclareSymbolFont{matha}{OML}{txmi}{m}{it}% txfonts
\DeclareMathSymbol{\varv}{\mathord}{matha}{118}

\makeatletter
\newsavebox{\@brx}
\newcommand{\llangle}[1][]{\savebox{\@brx}{\(\m@th{#1\big\langle}\)}%
  \mathopen{\copy\@brx\kern-0.5\wd\@brx\usebox{\@brx}}}
\newcommand{\rrangle}[1][]{\savebox{\@brx}{\(\m@th{#1\big\rangle}\)}%
  \mathclose{\copy\@brx\kern-0.5\wd\@brx\usebox{\@brx}}}
\makeatother

\newcommand{\longhookrightarrow}{\lhook\joinrel\longrightarrow}

\setlist[enumerate,1]{label={\rm(\arabic*)}, ref={\rm\arabic*}}

\newcommand{\supth}[1]{\ensuremath{#1^{\mathrm{th}}}}
\newcommand{\supst}[1]{\ensuremath{#1^{\mathrm{st}}}}

\EpigaVolumeYear{10}{2026} \EpigaArticleNr{6} \ReceivedOn{March 26, 2025}
\InFinalFormOn{July 7, 2025}
\AcceptedOn{October 8, 2025}

\title{Refined curve counting with descendants and quantum mirrors}
\titlemark{Refined curve counting with descendants and quantum mirrors}

\author{Patrick Kennedy-Hunt}
\address{University of Cambridge, Faculty of Mathematics,
Wilberforce Road,
Cambridge CB3 0WA,
United Kingdom} 
\email{pfk21@cam.ac.uk}

\author{Qaasim Shafi}
\address{Heidelberg University, Mathematikon, 
Im Neuenheimer Feld 205, 
69120  Heidelberg, Germany}
\email{mshafi@mathi.uni-heidelberg.de}

\author{Ajith Urundolil Kumaran}
\address{Massachusetts Institute of Technology, Department of Mathematics, 
Simons Building, 
77 Massachusetts Avenue, Cambridge, MA 02139-4307, USA}
\email{ajith270@mit.edu}

\authormark{P.~Kennedy-Hunt, Q.~Shafi and A.~Urundolil Kumaran}
\AbstractInEnglish{Given a log Calabi--Yau surface $(Y,D)$, Bousseau has constructed a quantization of the mirror algebra of this pair. %\cite{bousseau2020quantum}.
We give a formula for structure constants of this quantization in terms of higher-genus descendant logarithmic Gromov--Witten invariants of $(Y,D)$. Our result generalises the weak Frobenius structure conjecture for surfaces to the $q$-refined setting, and is proved by relating these invariants to counts of quantum broken lines in the associated quantum scattering diagram.}


\acknowledgement{P.\,K.-H.\ was supported by an Early Career Fellowship at the University of Sheffield awarded by the London Mathematical Society, and Peterhouse, Cambridge. Q.\,S.\ was supported by UKRI Future Leaders Fellowship through grant number MR/T01783X/1 and by the starting grant `Correspondences in enumerative geometry: Hilbert schemes, K3 surfaces and modular forms', No.~101041491 of the European Research Council.}

\begin{document}

%%%%%%%%%%%%%%%%%%%%%%%%%%%%%%%
% Title page
%%%%%%%%%%%%%%%%%%%%%%%%%%%%%%%

%\removeabove{}
%\removebetween{}
%\removebelow{}

\maketitle

\begin{prelims}

\DisplayAbstractInEnglish

\bigskip

\DisplayKeyWords

\medskip

\DisplayMSCclass

\end{prelims}

%%%%%%%%%%%%%%%%%%%%%
% Table of Contents
%%%%%%%%%%%%%%%%%%%%%

\newpage

\setcounter{tocdepth}{1}

\tableofcontents

%%%%%%%%%%%%%%%%%%%%%
% Content begins here
%%%%%%%%%%%%%%%%%%%%%

\section{Introduction}

This paper lies at the intersection of logarithmic enumerative geometry and quantization in mirror symmetry. Logarithmic Gromov--Witten invariants are, to first approximation, counts of algebraic curves of fixed degree and genus in a smooth projective variety $Y$. These curves are required to have prescribed tangencies to a given divisor $D$, and pass through some collection of cycles. Similarly to ordinary Gromov--Witten invariants, they have a wide range of applications across algebraic geometry. These invariants and their generalisations have proved crucial in mirror symmetry constructions, see \cite{gross2021intrinsic, gross2015mirror}; they provide insights into moduli spaces, for example the moduli space of curves, see \cite{graber2005relative,ranganathan2024logdr}; they enjoy connections to tropical geometry, see \cite{bousseau2018tropical,MandelRuddat,NishSieb06,graefnitz2022tropical}; and they are a key tool for computing ordinary Gromov--Witten invariants via the degeneration formula, see \cite{pandharipande2017gromov,okounkov2009gromov,maulik2006topological}.

In this paper, we connect the higher-genus (descendant) logarithmic Gromov--Witten theory of log Calabi--Yau surfaces to a combinatorial object known as a \emph{quantum scattering diagram}. Scattering diagrams play a key role in mirror symmetry, see \cite{gross2011affine}, and also appear in the study of cluster algebras, see \cite{gross2018canonical}, and Bridgeland stability conditions, see \cite{bridgeland2017scattering}. On the one hand, due to the combinatorial nature of quantum scattering diagrams, our result provides a technique for computing these invariants. On the other hand, the quantum scattering diagram is directly connected to the \emph{quantum mirror}, the quantization of the Gross--Hacking--Keel mirror to this log Calabi--Yau surface, see \cite{gross2015mirror,bousseau2020quantum}, and our main result proves a generalisation of the weak Frobenius structure conjecture for surfaces (see Section~\ref{section : gross-siebert program}), a result in mirror symmetry formulated within the Gross--Siebert program.

Fix a tuple $m = (m_1,\dots,m_n)$ of elements in $\ZZ^2$. From this tuple one can form the following data: a log Calabi--Yau surface $Y_m$ and a quantum scattering diagram $S(\hat\fD_m)$. Also fix  a tuple $\vartheta = (r_1,\dots,r_s)$ of vectors in $\ZZ^2$. We will associate to this data two sets of enumerative invariants, one algebro-geometric and one combinatorial. We now explain these constructions and the associated invariants in more detail.

\subsection{Logarithmic Gromov--Witten invariants of $\boldsymbol{(Y_m,\partial Y_m)}$}

Much of the geometry of toric varieties can be encoded combinatorially through their fans. For a projective toric surface, the fan consists of vectors in $\ZZ^2$, called rays, each one corresponding to a component of the toric boundary divisor. For example, the fan of $\PP^2$ consists of three rays, each associated with one of the projective coordinate lines.

Let $\overline{Y}_m$ be a projective toric surface whose fan contains the rays $-m_1,\dots,-m_n$ (as well as $r_1,\dots, r_s$). We construct $Y_m$ by blowing up $\overline{Y}_m$ at $n$ points, one on each boundary divisor corresponding to $-m_i$. We take the divisor $\partial Y_m$ to be the strict transform of the toric boundary, which is an anticanonical divisor. Each vector $r_i$ prescribes a tangency condition for a curve in $Y_m$ with respect to $\partial Y_m$. Specifically, we will be interested in curves which hit the strict transform of the toric boundary divisor corresponding to $r_i$ with tangency the lattice length of $r_i$.

Logarithmic Gromov--Witten invariants of $(Y_m,\partial Y_m)$ are defined as integrals on the moduli space of stable logarithmic maps to $(Y_m,\partial Y_m)$; see \cite{chen2011stable,abramovich2011stable,gross2012logarithmic}. This moduli space is a compactification of the space of maps from smooth, genus $g$, $(s+1)$-marked curves to $Y_m$, of fixed degree, in our case with tangencies to $\partial Y_m$ at the first $s$ markings given by $\vartheta = (r_1,\dots,r_s)$. The integrand consists of classes on the moduli space that encode geometric constraints on the curve. These classes come in two flavours; they either require the image of the curve to pass through certain cycles in $Y_m$ or are \emph{tautological classes} constraining the geometry of the underlying source curve. 

For any genus $g \geq 0$ and tuple $p = (p_1,\dots,p_n)$ of non-negative integers such that $\sum_{i=1}^{n} p_i m_i = \sum_{j=1}^s r_j$, we have the logarithmic Gromov--Witten invariant $$N_{g,\vartheta}^{\beta_p} = \int_{[\mathsf{M}_{g,\vartheta}(Y_m,\beta_p)]^{\vir}}(-1)^g\lambda_g \ev^* (\mathsf{pt}) \psi^{s-2}.$$
The degree of the curves here is $\beta_p$, a curve class which realises the tangency conditions $\vartheta$ and hits the $\supth{i}$ exceptional divisor with multiplicity $p_i$. Here $\ev^* (\pt)$ is the pullback of the class of a point in $Y_m$ under the evaluation morphism $\ev \colon \mathsf{M}_{g,\vartheta}(Y_m,\beta_p) \rightarrow Y_m$ which takes a parametrised curve in $Y_m$ and outputs the image of its $\supst{(s+1)}$ marking in $Y_m$. This corresponds to demanding that the curve pass through a general point in $Y_m$. On the other hand, $\lambda_g$ and $\psi$ are examples of tautological classes; see Section~\ref{label : tautological}. The class~$\lambda_g$ relates to the geometry of the local log Calabi--Yau threefold $Y_m \times \AAA^1$; 
see Section~\ref{section : threefolds}. The $\psi$-classes, often referred to as \emph{descendants}, are most well known in the context of Witten's conjecture, see \cite{witten1991gravity}, which connects integrals of $\psi$-classes on the moduli space of curves to the KdV hierarchy; see \cite{kontsevich1992intersection}.

\subsection{Quantum broken line counts in quantum scattering diagrams}

Strominger, Yau, and Zaslow \cite{strominger1996mirror} proposed a picture of mirror symmetry where mirror Calabi--Yau manifolds should exhibit dual special Lagrangian torus fibrations over a common base. In order to encode ``instanton corrections'', Kontsevich and Soibelman \cite{kontsevich2006affine} posited that, at least in dimension two, constructing the mirror could be reduced to computing commutators of certain families of automorphisms of a two-dimensional torus. This can in fact be encoded in a combinatorial object called a scattering diagram (on the base), and in the surface situation, this was used to build the mirror to a log Calabi--Yau surface; see \cite{gross2015mirror}. Moreover, in \cite[Section 11.8]{kontsevich2006affine}, it is remarked that these scattering diagrams admit a natural $q$-deformation, relating instead to formal families of automorphisms of the two-dimensional \emph{quantum} torus, a natural non-commutative deformation of the two-dimensional algebraic torus. 

Both scattering and quantum scattering diagrams consist of half-lines, called rays, in $\RR^2$ equipped with certain functions called Hamiltonians. These encode families of automorphisms of the torus, respectively quantum torus. We consider the quantum scattering diagram, $S(\hat{\fD}_m)$, built from the initial rays $-\RR_{\geq 0}m_i$ for $i=1,\dots,n$ with certain Hamiltonians \eqref{eq : initial hamiltonians}, possibly with some additional rays to make the diagram \emph{consistent}; see Section~\ref{Section : Quantum scattering diagrams}. 

Drawing an analogy between $Y_m$ and $S(\hat{\fD}_m)$, we will be interested in counts of \emph{quantum broken lines} in $S(\hat{\fD}_m)$, in some sense the combinatorial counterpart to curves in $Y_m$. Quantum broken lines (see Definition~\ref{def : quantum broken line count} and Figure~\ref{Figure : Quantum scattering}) are piecewise-linear curves in $\RR^2$ which bend only at the rays of the quantum scattering diagram in a way controlled by the Hamiltonian of that ray. Each quantum broken line $\gamma$ has the following features:  
\begin{itemize}
    \item It has an asymptotic starting direction in $\mathbb{Z}^2$, oriented outward along the infinite segment.  
    \item It ends at some point $Q \in \mathbb{R}^2$. 
    \item It has a final direction $\varv(\gamma) \in \ZZ^2$, oriented outward from $Q$.
    \item It carries a final monomial $c(\gamma)$ depending on the parameters $t_1,\dots,t_n$ and the quantum pa\-ra\-meter~$q$.  
\end{itemize}

\begin{figure}
    \centering
    \includegraphics[width=11cm]{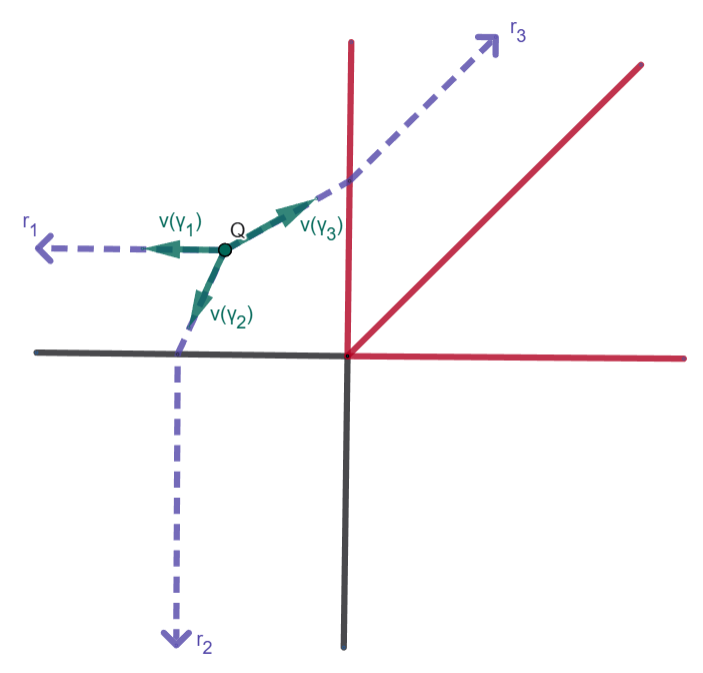}
    \caption{A pictorial representation of an example of a summand in \eqref{scattering invariant} for $m=((1,0),(0,1))$. The black rays are the initial rays of $S(\hat{\fD}_m)$; the red rays are rays added to make the diagram consistent (see \cite[Section 3.2.2]{bousseau2020quantum}). There are three broken lines $\gamma_1,\gamma_2,\gamma_3$, dashed, in purple. Each broken line ends at the balanced point $Q \in \RR^2$ and has asymptotic start directions $\vartheta = (r_1,r_2,r_3)$ and final directions $(\varv(\gamma_1),\varv(\gamma_2),\varv(\gamma_3))$. }
    \label{Figure : Quantum scattering}
\end{figure}

We define a combinatorial invariant, depending on our chosen tuple of vectors $\vartheta = (r_1,\dots, r_n)$, 
\begin{equation}\label{scattering invariant}
        \left\langle \hat{\vartheta}_{r_1},\dots, \hat{\vartheta}_{r_s}\right\rangle = \sum_{(\gamma_1,\dots,\gamma_s)} c(\gamma_1) \cdots c(\gamma_s) \prod_{2 \leq i < j \leq s} q^{\frac{1}{2}\varv(\gamma_i) \wedge \varv(\gamma_j)},   
\end{equation}
 where the sum is over tuples $(\gamma_1,\dots, \gamma_s)$ of quantum broken lines in the quantum scattering diagram $S(\hat{\fD}_m)$ with start directions $\vartheta = (r_1,\dots,r_s)$, meeting at a fixed, balanced point $Q \in \RR^2$, \textit{i.e.} such that $\sum_{i=1}^s \varv(\gamma_i) = 0$.

\begin{maintheorem}\label{thm:main}
    Fix tuples $m=(m_1,\dots,m_n)$ and $\vartheta = (r_1,\dots, r_s)$ of vectors in $\ZZ^2$. The generating function for the higher-genus descendant logarithmic Gromov--Witten invariants is determined by a combinatorial invariant involving counts of quantum broken lines. Precisely, there is an equality 
    $$\sum_{p}  \left(\sum_{g \geq 0} N_{g,\vartheta}^{\beta_p}u^{2g}\right)t_1^{p_1} \cdots t_n^{p_n} = \frac{1}{(s-1)!}\sum_{\omega \in \Omega_s}\left\langle \hat{\vartheta}_{r_{\omega(1)}},\ldots,\hat{\vartheta}_{r_{\omega(s)}} \right\rangle$$
after the change of variable $q= e^{iu}$. On the left-hand side, we sum over all $p=(p_1,\dots,p_n) \in \NN^n$ such that $\sum_{i=1}^{n} p_i m_i = \sum_{j=1}^s r_j$, and on the right-hand side, $\Omega_s$ denotes the set of cyclic permutations on $s$ elements.
\end{maintheorem}

Theorem~\ref{thm:main} suggests a method for computing higher-genus descendant logarithmic Gromov--Witten invariants of $(Y_m,\partial Y_m)$. The results are also linked to the GW/DT correspondence, see \cite{maulik2006gromov,maulik2006gromov2}, more specifically its logarithmic enhancement, see \cite{maulik2025logarithmicenumerativegeometrycurves} (for log Calabi--Yau surfaces). This correspondence suggests that two intersection-theoretic frameworks for curve counting, viewing curves through their parametrisations (Gromov--Witten theory) or through their defining equations (Donaldson--Thomas theory), should encode equivalent information. As far as we know, this correspondence has no link to mirror symmetry; however, an exponential change of variables as in Theorem~\ref{thm:main} also appears in the GW/DT correspondence. Moreover, from our theorem, one can see that the answer has rationality properties exactly predicted by the (conjectural) logarithmic GW/DT correspondence. If we take the conjecture as true, this gives the first computation of the descendant series of a non-toric target.

Furthermore, our results show that the, \textit{a priori} mysterious, agreement of descendant counts with structure constants from intrinsic surface mirrors of the Gross--Siebert program seems to persist to the $q$-refined setting. We explain this in more detail below.

\subsection*{Interpretation in the Gross--Siebert program}\label{section : gross-siebert program} The \emph{intrinsic mirror} construction of the Gross--Siebert program, see \cite{gross2021intrinsic,gross2022canonical}, assigns to a log Calabi--Yau pair $(X,D)$ a ring $R$, which is understood to be the ring of functions on the mirror. The authors conjecture, see~\cite[Conjecture 9.2]{gross2021intrinsic}, a connection, the \emph{weak Frobenius structure conjecture}, between the structure constants of the mirror algebra and the \emph{genus zero} descendant logarithmic Gromov--Witten theory of $(X,D)$. As both sides have connections to scattering diagrams, this is a plausible speculation. 
The weak Frobenius structure conjecture has since been proved in complete generality, see \cite{johnston2022comparison}, and has had applications, see \cite{mandel2019fanomirrorperiodsfrobenius,johnston2025quantumperiodstoricdegenerations}, to the Fanosearch program; see \cite{coates2013mirror}.

In this surface situation, Gross, Hacking, and Keel constructed a mirror equipped with the structure of a formal Poisson variety; see~\cite{gross2015mirror}. This is exactly the context in which one can speculate about the existence of a deformation quantization of the mirror algebra: a non-commutative algebra depending on a quantum parameter $q$ in which the non-commutativity of the product structure relates to the Poisson bracket. This deformation quantization has since been constructed; see \cite{bousseau2020quantum}. As alluded to, the right-hand side of Theorem~\ref{thm:main} can be interpreted as structure constants of this quantum mirror, and so our theorem connects structure constants of the \emph{quantum} mirror algebra to the \emph{higher-genus} descendant logarithmic Gromov--Witten theory of the surface.

Let $(Y,D)$ be a Looijenga pair (see Section~\ref{Section : Looijenga pairs}). Associated to this pair is a non-commutative $\mathbb{C}$-algebra given by the \emph{quantum mirror} to $(Y,D)$. This algebra has a basis given by $\hat{\vartheta}_r$ for $r$ an integral point in the tropicalisation $\mathsf{Trop}(Y,D) \cong \ZZ^2$ of $(Y,D)$. Let
$$\llangle \hat{\vartheta}_{r_1},\ldots,\hat{\vartheta}_{r_s}\rrangle$$ denote the coefficient of $1=\hat{\vartheta}_{(0,0)}$ in the algebra product of the basis elements $\hat{\vartheta}_{r_1},\ldots,\hat{\vartheta}_{r_s}$; see \cite[Theorem 5.1]{bousseau2020quantum}. Unlike in the classical setting, this bracket is not symmetric.

\begin{maincorollary}
    After averaging over cyclic permutations and making the change of variables $q=e^{iu}$, there is always an equality
    $$\sum_{\beta} z^{\beta} \left(\sum_{g \geq 0} N_{g,\vartheta}^{\beta}u^{2g}\right) = \frac{1}{(s-1)!}\sum_{\omega \in \Omega_s}\llangle \hat{\vartheta}_{r_{\omega(1)}},\ldots,\hat{\vartheta}_{r_{\omega(s)}} \rrangle.$$
\end{maincorollary}

This follows from Theorem~\ref{thm:main} due to the fact that the product structure on the quantum mirror algebra is defined in terms of the broken line counts appearing in \eqref{scattering invariant}. For more detail, see Section~\ref{Subsection : product structure}.

The proof of Theorem~\ref{thm:main} follows a higher-genus version of the strategy of the proof of the weak Frobenius structure conjecture in \cite{mandel2021theta}. We use the technique of factoring the quantum scattering diagram, see \cite{gross2010vertex}, as well as the connection between the Hamiltonians in the consistent quantum scattering diagram and tropical curves, see \cite{bousseau2020quantumtrop,filippini2015block}, to relate the right-hand side to tropical curve counts with an $s$-valent vertex. The refined multiplicity of the $s$-valent vertex appears due to the non-commutative product rule of the quantum torus, after symmetrising. We then use the degeneration formula in logarithmic Gromov--Witten theory, together with the refined descendant tropical correspondence theorem of \cite{kennedyhunt2023tropical} to connect the left-hand side to the same tropical curve counts. As in \cite{kennedyhunt2023tropical}, there are extra subtleties in the degeneration formula in this higher-genus, descendant setting. 

\subsection{Future directions}
\subsubsection{Threefolds}\label{section : threefolds} Theorem~\ref{thm:main} can be viewed as a result for the ``local log Calabi--Yau surface threefold", $Y_m \times \AAA^1$. The virtual dimension of the moduli space in the case of threefolds is independent of the genus, and invariants on $Y_m \times \AAA^1$, defined via virtual localisation, are exactly invariants of $Y_m$ with a $\lambda_g$ insertion. One future direction would be to generalise to the situation of honest threefolds, prompting questions about a higher-genus analogue of \cite{arguz2022higher}.
\subsubsection{K3 surfaces} In a different direction, one could combine Theorem~\ref{thm:main} with the use of type III degenerations of K3 surfaces to probe their enumerative geometry. The components of the special fibre of such a degeneration will be of the form $Y_m$, for example in the quartic K3 degeneration. The degeneration formula will be more subtle, see~\cite[Section 4]{maulik2010curves}, due to the presence of the reduced virtual class, particularly in the case of non-primitive classes, see \cite{blomme2024correlatedgromovwitteninvariants}. However, even in the case of primitive curve classes, if the scattering computations can be understood well, this provides an avenue to produce new calculations, for example, a generalisation of~\cite[Theorem 3]{maulik2010curves} with descendants.

\subsection{Related work}\begin{itemize}
    \item In \cite{bousseau2024looijenga}, the authors study five enumerative theories associated to a log Calabi--Yau surface. They propose correspondences relating logarithmic, local, and open Gromov--Witten invariants, as well as quiver DT invariants and BPS invariants associated to this log Calabi--Yau surface. In particular, in~\cite[Proposition 4.2]{bousseau2024looijenga}, a proof is sketched relating higher-genus invariants with a $\lambda_g$ insertion of log Calabi--Yau surfaces to counts of tuples of three quantum broken lines. 
    \item In \cite{grafnitz2025enumerative}, the authors study the geometry of a smooth Fano surface relative to a smooth anticanonical divisor. An interpretation for a certain quantum theta function is given in terms of a $q$-refinement of the open mirror map  defined by quantum periods of mirror curves for outer Aganagic--Vafa branes on the associated local threefold. The proof of \cite[Proposition 3.12]{grafnitz2025enumerative}, in particular, relates to Theorem~\ref{thm:main} and~\cite[Theorem A]{kennedyhunt2023tropical}.
\end{itemize}

\subsection{Acknowledgements}
We thank Sam Johnston, Dhruv Ranganathan, Yannik Schuler, and Michel
van Garrel for a number of helpful conversations. We thank Tim
Gr\"afnitz for helpful comments on a previous version. We thank Dhruv
Ranganathan and Calla Tschanz for helpful comments on the
introduction.

\section{Logarithmic Gromov--Witten theory and tropical curves}\label{Section : Statement + logGW}

\subsection{Looijenga pairs}\label{Section : Looijenga pairs}
The basic geometry in this paper is a Looijenga pair: that is a pair $(Y,D)$ consisting of a smooth complex projective surface $Y$, together with a simple normal crossings, singular, reduced, anticanonical divisor $D$ on~$Y$. 

\begin{proposition}[\textit{cf.} {\cite[Proposition 1.3]{gross2015mirror}}]
  Given a Looijenga pair $(Y,D)$, there exist, after potentially carrying out corner blow-ups, a toric surface $\overline{Y}_m$ and a blow-down morphism
  $$\nu\colon Y \longrightarrow \overline{Y}_m$$
  with $\nu$ the blow-up in a collection of points in the interior of the toric boundary.
\end{proposition}

Thus, to study all Looijenga pairs, it suffices to study the geometry described in Setting~\ref{setting:basic}. 

\begin{setting}\label{setting:basic}
    Fix a tuple $m = (m_1,\dots,m_n)$ of elements of $M = \ZZ^2$. After adding more elements if necessary to $m$ to form a tuple $m'$ which spans $\RR^2$, there is a unique two-dimensional fan $\Sigma_m$ with rays $-\RR_{\geq 0}m_j$ for $m_j \in m'$. We write $\overline{Y}_m$ for the toric surface with fan $\Sigma_m$. Write $D_{m_j}$ for the component of the toric boundary of $Y_m$ associated, under the toric dictionary, to the ray $-\RR_{\geq 0} m_j$.  
    
    For each $m_j$ for $j=1,\dots,n$, choose a point $x_j\in D_j$ such that $x_i \ne x_j$ whenever $i\ne j$, and define a surface~$Y_m$ to be the blow-up $\nu\colon Y_m \rightarrow \overline{Y}_m$ at the points $x_1,\dots,x_n$. Let $E_j$ denote the exceptional divisor over $x_j$ and $\partial Y_m$ the strict transform of the toric boundary divisor. Notice $(Y_m,\partial Y_m)$ is a Looijenga pair.
\end{setting}

In Setting~\ref{setting:basic}, a choice of tuple $m'$ extending $m$ was made. Since logarithmic Gromov--Witten invariants are insensitive to logarithmic modifications, see~\cite[Theorem 1.1.1]{AWbirational}, our main theorem is insensitive to this choice, and we fix a choice without further comment. 

\subsection{Logarithmic Gromov--Witten theory}
Let $\vartheta = (r_1,\dots r_s)$ be a tuple of vectors in $M=\ZZ^2$, and replace $m'$ by $m'~\cup~\vartheta$ so that every $\RR_{\geq 0}r_i$ is a ray in the fan of $\overline{Y}_m$. We say that a map from an $(s + 1)$-pointed curve to ${Y}_m$ has \emph{tangency given by $\vartheta$} if the $\supth{i}$ marked point hits the strict transform of the toric boundary divisor corresponding to $r_i$ with tangency given by the lattice length of $r_i$, in other words, the maximal positive integer $\ell$ such that one can write $r_i = \ell \overline{r}_i$ for $\overline{r}_i$ a vector in $M$, called the \textit{direction} of $r_i$.

The moduli space parametrising $(s+1)$-pointed genus $g$ stable maps to $Y_m$ with curve class $\beta$ and tangency given by $\vartheta$ at the first $s$ markings is not proper. The moduli space of stable logarithmic maps $\mathsf{M}_{g,\vartheta}(Y_m,\beta)$ is a compactification; see \cite{abramovich2011stable,chen2011stable,gross2012logarithmic}. 

Write $\overline{\mathcal{M}}_{g,s+1}$ for the moduli space of stable genus $g$ curves with $s+1$ marked points. This space comes equipped with universal curve $\mathsf{pr}\colon \mathcal{C} \rightarrow \overline{\mathcal{M}}_{g,s+1}$ and a forgetful morphism
$$\pi'\colon\mathsf{M}_{g,\vartheta}(Y_m,\beta) \longrightarrow \overline{\mathcal{M}}_{g,s+1}.$$
The moduli space of stable curves carries two flavours of tautological bundle of import to us: 
\begin{itemize}\label{label : tautological}
\item The first is the \textit{Hodge bundle} $\mathbb{E}_g = \mathsf{pr}_\star \omega_p$, where $\omega_p$ is the relative dualising sheaf of $p$; we write $\lambda_g = c_g(\mathbb{E}_g)$. 
\item Note that $\overline{\mathcal{M}}_{g,s+1}$ carries $s+1$ tautological sections identifying the marked points. Denote the last section by $S$, and define
  $$\psi = c_{1}\left(S^\star \omega_p\right).$$ 
\end{itemize} 
Both classes can be pulled back along $\pi'$ to define tautological classes on $\mathsf{M}_{g,\vartheta}(Y_m,\beta)$ which, by abuse of notation, we also call  $\psi$ and $\lambda_g$. 

We have an evaluation morphism $\ev \colon \mathsf{M}_{g,\vartheta}(Y_m,\beta) \rightarrow Y_m$ which evaluates at the $\supst{(s+1)}$ marked point. 

\subsubsection{Invariants}\label{Section : invariants}
The moduli space $\mathsf{M}_{g,\vartheta}(Y_m,\beta)$ carries a virtual fundamental class $\left[\mathsf{M}_{g,\vartheta}(Y_m,\beta)\right]^\mathsf{vir}$ allowing us to define \textit{logarithmic Gromov--Witten invariants}. We will consider the following descendant logarithmic Gromov--Witten invariants with a $\lambda_g$ insertion:
$$\mathsf{N}_{g,\vartheta}^{\beta} = \int_{\left[\mathsf{M}_{g,\vartheta}(Y_m,\beta)\right]^\mathsf{vir}}(-1)^g\lambda_g \ev^* (\mathrm{pt}) \psi^{s-2}.$$

The classes $\lambda_g$ and $\psi$ could instead be defined directly on the moduli space $\mathsf{M}_{g,\vartheta}(Y_m,\beta)$; we call these classes $\overline{\lambda}_g$ and $\overline{\psi}$.

\begin{proposition} We have that 
    $$\overline{\lambda}_g = \lambda_g \quad \text{and}\quad  \mathsf{N}_{g,\vartheta}^{\beta} = \int_{\left[\mathsf{M}_{g,\vartheta}(Y_m,\beta)\right]^\mathsf{vir}}(-1)^g\overline{\lambda}_g \ev^* (\mathrm{pt}) \overline{\psi}^{s-2}.$$
\end{proposition}

\begin{proof}
    The first equality follows from~\cite[Remark 2.1]{kennedyhunt2023tropical}. For the second equality, we can use~\cite[Example 6.3.4(a)]{Fulton} to rewrite $\mathsf{N}_{g,\vartheta}^{\beta}$ as an integral over $\mathsf{M}_{g,\vartheta}(Y_m,\beta,P)$, where we choose a generic point $P$ in the interior that the marking evaluates to. Now the difference $\psi - \overline{\psi}$ on $\mathsf{M}_{g,\vartheta}(Y_m,\beta,P)$ is supported on the locus of curves where the component containing this marking is destabilised under the forgetful morphism. We claim this locus is empty. By the argument of~\cite[Proposition 3.4]{MandelRuddat}, it suffices to show that when this component $C'$ is rational, it contains at least three special points. If $C'$ contains two or more additional markings, then it will not be destabilised under the forgetful morphism. 
    If $C'$ contains just one other marking, then it will also have to contain a node, because there are necessarily other relative markings which need to be distributed on some component. 
    So the only possibility is that $C'$ does not hit the boundary in any point. By composition with the map $\nu\colon Y_m \rightarrow \overline{Y}_m$, the map from $C'$ to $Y_m$ defines a curve in a toric surface with fixed incidence to the toric boundary. Counting dimensions of the space of maps from rational curves to toric surfaces with fixed incidence to the toric boundary, we deduce this does not happen for a generic choice of $x_i$ and $P$.
\end{proof}

The class $\beta_p$ of a curve is specified by two data: 
\begin{enumerate}
    \item the tuple $p = (p_1,\dots,p_n)$ recording, for each exceptional curve $E_i$, the intersection product $\beta \cdot E_i = p_i$, 
    \item the tangency profile $\vartheta$.
\end{enumerate}
In the toric situation, we consider more general invariants; consider a balanced multiset $\Delta^{\circ}$ of vectors in $\ZZ^2\setminus \{(0,0)\}$ and non-negative integers $g$ and $n$. As in~\cite[Section 2.1]{kennedyhunt2023tropical}, we have associated to this multiset the moduli space $\mathsf{M}_{g,\Delta} = \mathsf{M}_{g,\Delta}(X_{\Delta},\beta_{\Delta})$ of genus $g$, ($|\Delta^{\circ}| + n$)-pointed stable logarithmic maps to the toric variety $X_{\Delta}$ with tangency given at the last $|\Delta^{\circ}|$ markings by $\Delta^{\circ}$. Fix  a subset $\Delta_F$ of $\Delta^{\circ}$, and let
$$ \ev_{\Delta_F} \colon \mathsf{M}_{g,\Delta} \longrightarrow (\partial X_{\Delta})^{|\Delta_F|}$$
denote the evaluation morphism at the toric boundary divisor $\partial X_{\Delta}$ indexed by the elements of $\Delta_F$. Let
$$\ev_i \colon \mathsf{M}_{g,\Delta} \longrightarrow X_{\Delta}$$
denote the evaluation morphism at the non-tangency markings $i=1,\dots,n$. Fix non-negative integers $k_1,\dots,k_n$ such that
$$n-1 + |\Delta^{\circ}| = 2n + \sum_{i=1}^n k_i + |\Delta_F|,$$
and let
$$\mathsf{N}_{g,\Delta,\Delta_F}^{\bf{k}} = \int_{[\mathsf{M}_{g,\Delta}]^{\vir}} (-1)^{g} \lambda_{g} \prod_{i=1}^{n} \ev_{i}^{*}(\mathsf{pt})\psi^{k_i} \ev_{\Delta_F}^* \left(r^{|\Delta_{F}|}\right),$$
where $r$ is the class of a point on $\partial X_{\Delta}$ and $r^{|\Delta_{F}|}$ denotes the class $\prod_i\pi_i^\star (r)$ in $A^\star \left((\partial X_{\Delta})^{|\Delta_F|}\right)$, where $\pi_i$ denotes the $\supth{i}$ projection $\pi_i \colon (\partial X_{\Delta})^{|\Delta_F|} \rightarrow \partial X_{\Delta}$.

\subsection{Tropical curve counting problems}

\subsubsection{Tropical curves} We briefly recall notation for tropical curves from \cite{kennedyhunt2023tropical}; see also \cite[Section~2.3]{bousseau2018tropical}. We refer the reader to \cite{abramovich2020decomposition,MandelRuddat,Milk,NishSieb06} for background.

    Following \cite{abramovich2016skeletons}, a \textit{graph} $\Gamma = (V(\Gamma),E_f(\Gamma),E_{\infty}(\Gamma))$ is a triple consisting of finite sets of vertices $V(\Gamma)$, bounded edges $E_f(\Gamma)$, and a multiset $E_{\infty}(\Gamma)$ of unbounded edges. We assume all graphs are connected. 
An \textit{abstract tropical curve} $|\Gamma|$ is the underlying topological space of a graph $\Gamma$.

{\samepage
  \begin{definition}
    A parametrised tropical curve ${h}\colon{\Gamma} \rightarrow \mathbb{R}^2$ consists of the following data:
    \begin{enumerate}
    \item a graph $\Gamma$ and a non-negative integer $g_V$ assigned to each vertex $V$ of $\Gamma$, called the \textit{genus};
    \item a bijective function
          $$L\colon E_\infty(\Gamma)\longrightarrow \{1,\ldots,r+n\};$$
        \item a \textit{vector weight} $v_{V,E} \in \mathbb{Z}^2$ for every edge-vertex pair $(V,E)$ with $E\in E_f(\Gamma)\cup E_\infty(\Gamma)$ and $V\in E$ such that for every vertex $V$, the
following balancing condition is satisfied:
$$\sum_{E: V \in E} v_{V,E}=0;$$
\item for each bounded edge $E \in E_f(\Gamma)$, a positive real number $\ell(E)$, called the \textit{length} of $E$;
\item a map of topological spaces $h\colon |\Gamma| \rightarrow \mathbb{R}^2$ such that restricting $h$ to the edge $\{v_1,v_2\}$ is affine linear to the line segment connecting $h(V_1)$ and $h(V_2)$ and moreover 
$$h(V_2) - h(V_1) = \ell(E)v_{V_1,E}.$$
Also, restricting $h$ maps an unbounded edge $E$ associated to a vertex $V$ to the ray $h(V) + \mathbb{R}_{\geq 0}v_{V,E}$.
    \end{enumerate} 
We say $h$ has \textit{degree} $\Delta$ if $v_{V,E}$ coincides with the $\supth{L(E)}$ column of $\Delta$ whenever $E \in E_\infty (\Gamma)$. The \textit{genus} of a parametrised tropical curve is obtained by adding the sum of the $g_V$ to the Betti number of $|\Gamma|$. The \textit{weight} of an edge $E$, denoted by $w(E)$, is the lattice length of $v_{V,E}$.

Fixing a subset $\Delta^F$ of $\Delta$ and a generic configuration $\{x_v\}_{v\in \Delta_F}$ of points in $\mathbb{R}^2$, we say $h$ has \textit{degree} $(\Delta,\Delta^F)$ if it has degree $\Delta$ and the unbounded edges in correspondence
with $\Delta^F$ asymptotically coincide with the half-lines $x_v + \RR_{\geq 0}v$ for $v \in \Delta^F$.
  \end{definition}
  }

For a vertex $V$ of $\Gamma$, write $E_\infty^+(V)$ for the set of unbounded edges $E$ adjacent to $V$ such that $v_E\ne 0$ and $E_f(V)$ for the set of bounded edges adjacent to $V$. The \textit{valency} $\mathsf{val}_V$ of a vertex $V$ is the cardinality of $E_f(V)\cup E_\infty^+(V)$. Write $\Delta_V^\circ$ for the multiset of all $v_{V,E}$ for a fixed $V$.

\subsection{Multiplicities}\label{sec: multiplicities}
We will now count parametrised genus zero tropical curves of degree $(\Delta,\Delta^F)$ satisfying certain incidence conditions. For the remainder of the section, fix  a parameterised tropical curve ${h}\colon\Gamma \rightarrow \mathbb{R}^2$. Tropical curves are counted with a multiplicity, closely related to the multiplicity of \cite{blechman2019refined}. This multiplicity is given as a product\looseness=-1
$$m_{h}\left(q^{\frac{1}{2}}\right) = \prod_{V \in V(\Gamma)}m_V\left(q^{\frac{1}{2}}\right)$$
over multiplicities $m_V$ assigned to each vertex $V$ of our tropical curve. 
We will be counting tropical curves passing through a tuple of points $p=(p_1,\ldots,p_n)$ in $\mathbb{R}^2$, and thus vertices of $\Gamma$ come in two flavours. A vertex is \textit{pointed} if its image under $h$ coincides with one of the $p_i$. Vertices which are not pointed are \textit{unpointed}.

We define $v_1\wedge v_2$ to be the determinant of the matrix with first column $v_1$ and second column $v_2$. The cyclic group with $m$ elements acts on the set of ordered tuples of $m$ distinct elements from the set $\{1,\ldots,m\}$. 
The action is induced by sending the integer in position $i$ to position $i+1$ mod $m$. The set of orbits of this action is the set $\Omega_m$ of \textit{cyclic permutations}. For a cyclic permutation $\omega$,  choose an ordered tuple $\tilde{\omega}$ in the orbit of $\omega$. Fix a balanced tuple of vectors $(a_1,\ldots,a_m)$. Define
$$k(\omega) = \sum_{2\leq i < j \leq m}a_{\tilde{\omega}(i)} \wedge a_{\tilde{\omega}(j)},$$
where $\tilde{\omega}(i)$ sends $i$ to the element in the $\supth{i}$ position of the chosen representative $\tilde{\omega}$. As the vectors $a_i$ have sum zero, $k(\omega)$ is well defined. 

\begin{definition}\label{Definition : blechman-shustin} Let $V$ be a pointed vertex of valency $m$ given by vectors $a_1,\dots,a_m \in \ZZ^2$. The multiplicity of $v$ is $$m_V = \frac{1}{(m-1)!}\sum_{\omega \in \Omega_N} q^{\frac{k(\omega)}{2}}.$$

If, however, $V$ is an unpointed trivalent vertex with the balanced set of vectors $(a_1,a_2,a_3)$, then the multiplicity assigned to $V$ is the Block--G\"ottsche multiplicity, see \cite{block2016refined}:  
\begin{equation}\label{eq : BG multiplicity}
m_V = \frac{q^{\frac{1}{2}|a_1 \wedge a_2|} - q^{-\frac{1}{2}|a_1 \wedge a_2|}}{q^{\frac{1}{2}} - q^{-\frac{1}{2}}}.
\end{equation}
The multiplicity of a bivalent vertex is $1$.
\end{definition}

\subsection{Tropical counting problem} We fix once and for all a subset $\Delta^F$ of $\Delta$ and a generic configuration $\{x_v\}_{v\in \Delta_F}$ of points in $\mathbb{R}^2$. Thus it will make sense to discuss tropical curves of degree $(\Delta,\Delta^F)$. For a generic tuple $p=(p_1,\dots,p_n)$  of $n$ points in $\mathbb{R}^2$, let $T_{\Delta,\Delta^F,p}^\bk$ be the set of rigid parametrised tropical curves  of degree $(\Delta,\Delta^F)$ passing through $p$ with degree $\bk$. This means $h(E_i) = p_i$ for $i = 1, \dots , n$, where $E_1,\dots,E_n$ are the last $n$ unbounded edges. Moreover, $E_i$ is attached to a vertex of valency at least $k_i + 2$.

\begin{proposition}\label{prop:TropCurveCount}
There is an open dense subset $U_{n}^\bk(\Delta_F)$ of\, $\mathbb{R}^{2n}\times \mathbb{R}^{2|\Delta_F|}$ such that if $p, (x_v)_{v\in \Delta^F}\in U_{n}^\bk(\Delta_F)$, then $T_{\Delta,\Delta^F,p}^\bk$ is a finite set and the valency of the vertex supporting the unbounded edge $E_i$ is $k_i + 2$. Moreover, we may choose $U_{n}^{\bk}(\Delta_F)$ such that all parametrised tropical curves passing through $p$ with degree $\bk$ are rigid.
\end{proposition}

\begin{proof}
    The argument of \cite[Proposition 1.4.1]{kennedyhunt2023tropical} works with a cosmetic modification.
\end{proof}

From now on we assume $p_i \neq p_j$ whenever $i$ and $j$ are distinct without further comment.

\begin{remark}
Since there are only finitely many combinatorial types of rigid parametrised tropical curves of degree $\Delta$, it is automatic that the set $T_{\Delta,p}^\bk$ is finite.
\end{remark} 

Recall the notation $m_V$ for the multiplicity of the vertex $V$ defined in Section~\ref{sec: multiplicities}.
 
\begin{definition}\label{defn:TropCurveCount}
  Fix $ (p,x) = ((p_1,\ldots,p_n),(x_v)_{v\in \Delta^F})$ in $U_n^\bk(\Delta_F)$, and define $$ \mathsf{N}^{\Delta,\bk}_{\mathsf{trop},\Delta_F}(q) =\sum_{h \in T_{\Delta,\Delta',p}^\bk} \prod_{V \in V(\Gamma)}m_V(q).$$
Also define  $\Ntrop(1)=\Ntrop$.
\end{definition}

\textit{A priori} the count $\Ntrop(q)$ depends on the choice of the point in $U_{n}^\bk(\Delta_F)$. We suppress this dependence from our notation as it is independent \textit{a posteriori}.

\subsection{The tropical counting and logarithmic Gromov--Witten theory}
Fix a subset $\Delta^F \subset \Delta$. 

\begin{theorem}\label{thm : bdyrefined} After the change of variables $q=e^{iu}$, there is an equality of generating series
$$\sum_{g \geq 0} \mathsf{N}_{g,\Delta,\Delta_F}^{\bf{k}}u^{2g -2 + |\Delta^{\circ}| - \sum k_i} =  \left(\prod_{v\in \Delta^F}\frac{1}{|v|}\right)\mathsf{N}^{\Delta,\bf{k}}_{\mathsf{trop},\Delta_F}(q)\left((-i)\left(q^{\frac{1}{2}} - q^{-\frac{1}{2}}\right)\right)^{|\Delta^{\circ}| - \sum_i k_i - 2}.$$ 
\end{theorem}

\begin{proof}
    The theorem is almost identical to~\cite[Theorem A]{kennedyhunt2023tropical}. Indeed, applying the degeneration formula yields the analogue of~\cite[Proposition 3.4.2]{kennedyhunt2023tropical}. In absorbing the contribution of weights of bounded edges as in~\cite[Section 6.1]{kennedyhunt2023tropical}, one must divide through by the weights of the fixed unbounded edges to compensate, which explains the prefactor $|v|^{-1}$ of the right-hand side of the statement. Since the vertex contributions are unaffected, the theorem then follows from the vertex contribution calculations in~\cite[Section 6.2]{kennedyhunt2023tropical}.
\end{proof}

\section{Scattering diagrams and tropical curves}\label{Section : scattering}
Let $M = \Hom(T,\CC^*) \cong \ZZ^2$ be the character lattice of a two-dimensional torus $T$. We will make use of the ``algebra of functions'' of the quantum torus $$\CC_q[M] = \bigoplus_{m \in M} \CC\left[q^{\pm \frac{1}{2}}\right] \hat{z}^m,$$ the $\CC[q^{\pm \frac{1}{2}}]$ algebra with non-commutative product given by $\hat{z}^{r} \cdot \hat{z}^{r'} = q^{\frac{1}{2}r \wedge r'}\hat{z}^{r+r'}$ for $r,r' \in M$, as well as the change of variables $q = e^{iu}$.

\subsection{Quantum scattering diagrams}\label{Section : Quantum scattering diagrams}
Let $R$ be a complete local $\mathbb{C}$ algebra with maximal ideal $\mathbf{m}_R$. A quantum scattering diagram over a ring $R$ is the data of a collection of half-lines $\{\mathfrak{d}\}$ called \textit{rays} in $M_{\RR} \cong \mathbb{R}^2$ with primitive integral direction $m_{\mathfrak{d}} \in M\setminus \{0\}$. To each ray is associated a function called the \textit{Hamiltonian} $H_{\mathfrak{d}}$ such that
$$H_{\mathfrak{d}} \in \varprojlim_{\ell} \left(R/\mathbf{m}_{R}^{\ell} \otimes \CC (\!(u)\!)\left[\hat{z}^{m_{\mathfrak{d}}}\right]\right).$$
See \cite[Section 1]{bousseau2020quantumtrop} for further details on quantum scattering diagrams.

The Looijenga pair $(Y_m, \partial Y_m)$ has an associated quantum scattering diagram called the \textit{initial quantum scattering diagram}.

\begin{definition}
The \textit{initial quantum scattering diagram} $\hat{\fD}_{\fm}$ associated to the tuple $m$ is the quantum scattering diagram over $R:=\mathbb{C}\llbracket t_1, \ldots, t_n \rrbracket$
consisting of 
incoming rays 
$(\fd_j, \hat{H}_{\fd_j})$ for $1 \leqslant j \leqslant n$, where 
\[ \fd_j = - \mathbb{R}_{\geqslant 0} m_j \]
and 
\begin{equation}\label{eq : initial hamiltonians}
\hat{H}_{\fd_j} =
\sum_{\ell \geqslant 1}\frac{1}{\ell} \frac{(-1)^{\ell-1}}{q^{\frac{\ell}{2}}-q^{-\frac{\ell}{2}}} t_j^\ell \hat{z}^{\ell m_j} 
\end{equation}
with $q = e^{iu}$.
\end{definition}

Associated to any scattering diagram is its \textit{consistent completion}; see~\cite[Theorem 6]{kontsevich2006affine}.

\begin{definition}\label{Definition : scattering completion}
    Let $S(\hat{\fD}_{\fm})$ be the quantum scattering diagram obtained by forming the \textit{consistent completion} of $\hat{\fD}_{\fm}$, as in \cite{kontsevich2006affine}. 
\end{definition}  

Theorem~\ref{thm:main} relates logarithmic Gromov--Witten invariants of $(Y_m,\partial Y_m)$ to quantum broken line counts on $S(\hat{\fD}_{\fm})$; see Section~\ref{sec : broken lines}. These counts also determine the product structure on the quantum mirror algebra to $(Y_m,\partial Y_m)$; see Section~\ref{Subsection : product structure}.

\subsection{Factoring the quantum scattering diagram}
The consistent quantum scattering diagram $S(\hat{\fD}_{\fm})$ is combinatorially complex, but it can be understood in terms of \textit{factored quantum scattering diagrams} whose structure reflects the combinatorics of tropical curves. This connection was established in \cite{filippini2015block}. We follow~\cite[Section 6]{bousseau2020quantumtrop}.

To factor a scattering diagram over the ring $R = \mathbb{C} \llbracket t_1,\ldots,t_n \rrbracket$, we study a scattering diagram over the ring $R_N = \CC[t_1,\dots,t_n]/(t_1^{N+1},\dots,t_n^{N+1})$. There is an embedding
$$R_N \longhookrightarrow \tilde{R}_N = \CC[u_{ij} : 1 \leq i \leq n, 1 \leq j \leq N]\,/\langle u_{ij}^2 :  1 \leq i \leq n, 1 \leq j \leq N \rangle $$
by
$$t_{i} = \sum_{j=1}^N u_{ij}, \,\,\, \Longrightarrow \,\,\, t^{\ell}_{i} = \sum_{A \subset \{1,\dots,N\}, |A| = \ell} \ell! \prod_{a \in A} u_{i a}.$$

Under this embedding we have 
\[
\hat{H}_{\fd_j} =
\sum_{\ell = 1}^N \sum_{A \subset \{1,\dots,N\},|A|=\ell}\left(\frac{1}{\ell} \frac{(-1)^{\ell-1}}{q^{\frac{\ell}{2}}-q^{-\frac{\ell}{2}}}\right) \ell! \left(\prod_{a \in A}u_{ja} \right)\hat{z}^{\ell m_j}.
\]
We then define a new initial quantum scattering diagram $\hat{\fD}_{\fm}^{\mathrm{split}}$ with initial rays $(\fd_{j \ell A}, \hat{H}_{j \ell A})$ for $\ell$ in  $\{1,\dots, N\}$ and $A \subset \{1,\dots, N\}$ with $|A| = \ell$, where $$\fd_{j \ell A} = -\RR_{\geq 0}m_j + c_{j \ell A} $$ for some generic $c_{j \ell A} \in \RR^2$ in general position and
\[
\hat{H}_{\fd_{j \ell A}} =
\left(\frac{1}{\ell} \frac{(-1)^{\ell-1}}{q^{\frac{\ell}{2}}-q^{-\frac{\ell}{2}}}\right) \ell! \left(\prod_{a \in A}u_{ja} \right)\hat{z}^{\ell m_j}.
\]
By~\cite[Lemma 6.1]{bousseau2020quantumtrop}, the Hamiltonian attached to \emph{any} ray in $S(\hat{\fD}_{\fm}^{\mathrm{split}})$ can be determined recursively from the initial rays and Hamiltonians.  

\subsection{From factored scattering diagrams to tropical curves}

\begin{lemma}[\textit{cf.} {\cite[Proof of Proposition 6.2]{bousseau2020quantumtrop}}]\label{Lemma : ray bijection}
For a general choice of\, $\hat{\fD}_{\fm}^{\mathrm{split}}$, there is a bijective cor\-re\-spon\-dence between $(\fd,H_{\fd}) \in S(\hat{\fD}_{\fm}^{\mathrm{split}})$ and rational tropical curves $h \colon \Gamma \rightarrow \RR^2$ such that the following hold:  
\begin{enumerate}[label={\rm(\roman*)}, ref={\rm\roman*}]
\item There is an edge $E_{\mathsf{out}}\in E_{\infty}(\Gamma)$ with
$h(E_{\mathsf{out}})=\fd$.
\item If\, $E\in E_{\infty}(\Gamma)\setminus
\{E_{\mathsf{out}}\}$ or if\, $E_{\mathsf{out}}$ is the only edge of\, $\Gamma$
and $E=E_{\mathsf{out}}$, 
then $h(E)$ is contained in some $\fd_{j \ell A}$ where 
$$1\le j\le n,\quad A\subseteq\{1,\ldots,N\}, \quad \ell\ge 1.$$
Furthermore,
if\, $E\not=E_{\mathsf{out}}$, then  
the unbounded direction of\, $h(E)$ is given by $-m_j$.
\item
 If\, $E,E'\in E_{\infty}(\Gamma)\setminus\{E_{\mathsf{out}}\}$
and $h(E)\subseteq \fd_{j\ell A}$ and $h(E')\subseteq
\fd_{j \ell' A'}$, then $A\cap A'=\emptyset$.
\item If\, $E\in E_{\infty}(\Gamma)\setminus\{E_{\mathsf{out}}\}$
or if $E_{\mathsf{out}}$ is the only edge of\, $\Gamma$ and $E=E_{\mathsf{out}}$,
and $h(E)\subseteq \fd_{j \ell A}$, we have $w_{\Gamma}(E)=\ell$.
\end{enumerate}
Furthermore,  we have that
\begin{equation}\label{eq : hamiltonian expression}
\hat{H}_{\fd} = m_{\Gamma}(q^{\frac{1}{2}})\prod_{j=1}^n \prod_{\ell \geq 1}\left(\frac{(-1)^{\ell}}{\ell} \frac{q^{\frac{1}{2}} - q^{-\frac{1}{2}}}{q^{\frac{\ell}{2}} - q^{-\frac{\ell}{2}}}\right)^{k_{\ell,j}} (\ell !)^{k_{\ell,j}} \left(\prod_{A \in A^{\Gamma}_{j \ell}} \prod_{a \in A} u_{ja}\right)\frac{\hat{z}^{\ell_p m}}{q^{\frac{1}{2}} - q^{-\frac{1}{2}}}, 
\end{equation}
where $m_{\Gamma}(q^{\frac{1}{2}})$ is the refined multiplicity of the tropical curve $h\colon \Gamma \rightarrow \RR^2$, $A_{j \ell}^{\Gamma}$ is the set of subsets of\, $\{1,\dots, N\}$ of size $\ell$ such that $\fd_{j \ell A}$ is an unbounded ray of $h \colon\Gamma \rightarrow \RR^2$, and $k_{\ell j} = |A_{j \ell}^\Gamma|$. Furthermore, $m$ is the primitive direction of the ray $\fd$, and $\ell_p$ is such that $$\sum_{j=1}^n \sum_{\ell \geq 1} \ell k_{\ell,j} m_j = \ell_p m.$$
\end{lemma}

\subsection{Monomials on broken lines}\label{sec : broken lines}

\begin{definition}
    A \emph{quantum broken line} is a pair $(\gamma,\{m_L\})$ consisting of  
    \begin{enumerate}
    \item a proper continuous piecewise integral affine map
      $$\gamma \colon (-\infty,0] \longrightarrow M_{\RR}$$ with only finitely many domains of linearity, 
        \item for each $L \subset (-\infty ,0]$ that is a maximal connected domain of linearity of $\gamma$, a choice of monomial $m_{L} = c_L \hat{z}^{p_L}$ with $c_L \in R_{u} = R \llbracket u \rrbracket$ and $p_L \in M$,
    \end{enumerate} 
    where these data satisfy the following:  
\begin{itemize}
    \item Along any domain of linearity, the direction is given by $-p_L$.
    \item $\gamma(0) = Q \in M_{\RR}$.
    \item For the unique unbounded domain of linearity $L$, we have that $m_L = \hat{z}^{r}$ for some $r \in M \setminus \{0\}$. 
    \item Let $t \in (-\infty,0)$ be a point of $\gamma$ which is not linear, passing from the domain of linearity $L$ to the domain of linearity $L'$. Then $\gamma$ necessarily crosses a ray $\fd$ in the scattering diagram at time $t$. Moreover, $c_{L'} \hat{z}^{p_{L'}}$ is a summand of the image of $ c_L \hat{z}^{p_L}$ under the automorphism
      $$c_L \hat{z}^{p_L} \longmapsto \exp\left(\hat{H}_{\fd}\right)c_L \hat{z}^{p_L}\exp\left(-\hat{H}_{\fd}\right)$$ 
\end{itemize}
\end{definition}

We denote the monomial attached to the segment ending at $Q$ by $c(\gamma)\hat{z}^{\varv(\gamma)}$.

\begin{definition}\label{def : quantum broken line count}
Let $\vartheta = (r_1,\dots,r_s)$ be a tuple of vectors in $M=\ZZ^2$. Define 
\begin{equation}\label{quantum broken line count}
\left\langle \hat{\vartheta}_{r_1},\dots, \hat{\vartheta}_{r_s}\right\rangle = \sum_{(\gamma_1,\dots,\gamma_s)} c(\gamma_1) \dots c(\gamma_s) \prod_{2 \leq i < j \leq s} q^{\frac{1}{2}\varv(\gamma_i) \wedge \varv(\gamma_j)}, 
\end{equation}
where the sum is over tuples $(\gamma_1,\dots,\gamma_s)$  of broken lines such that $\gamma_i(0) = Q$ for some fixed $Q$, $m_L = \hat{z}^{r_i}$, and $\sum_{i=1}^{s} \varv(\gamma_i) = 0$.
\end{definition}

The notation $\hat{\vartheta}_{r_i}$ comes from the connection to the product structure on the quantum mirror. Let $R = \CC\llbracket t_1,\dots, t_n \rrbracket$. 
\subsection{Product structure}\label{Subsection : product structure}

In this section, for simplicity, we assume that $(Y_m,\partial Y_m)$ is positive, see \cite[Definition 6.10]{gross2015mirror}, in order for the algebra structure to make sense without completing (see for example \cite[Proof of Corollary 4.7]{mandel2021theta}). For the general case, see \cite{bousseau2020quantum}. Let
$$R_{(Y_m,\partial Y_m)} = \bigoplus_{r \in M}R_u\hat{\vartheta}_r$$
be the free $R_u$-module. We call the basis elements $\hat{\vartheta}_r$ \emph{quantum theta functions}. 
We explain briefly how \eqref{quantum broken line count} relates to a product structure on $R_{(Y_m,\partial Y_m)}$. 
Fix a point $Q \in \RR^2 \setminus S(\hat{\fD}_{\fm})$. Let
$$\hat{\vartheta}_{r}^{Q} = \sum_{\gamma} c(\gamma) \hat{z}^{\varv(\gamma)} \in R_u \otimes_{\CC[q^{\pm \frac{1}{2}}]} \CC_q[M],$$
where the sum is over all broken lines $\gamma$ that have start direction given by $r$ and end at the point $Q$. We define a product structure on $R_{(Y_m,\partial Y_m)}$ via the structure constants
$$\hat{\vartheta}_{r_1}\cdots\hat{\vartheta}_{r_s} = \sum_{r}c^{r}_{r_1 \cdots r_s} \hat{\vartheta}_r,$$
where $c^{r}_{r_1 \cdots r_s}$ is given by the coefficient of $\hat{\vartheta}_{r}^Q$ in $\hat{\vartheta}^Q_{r_1}\cdots\hat{\vartheta}^Q_{r_s}$ for $Q$ sufficiently close to $r$.

\begin{lemma}\label{Lemma : structureconstant}
  We have that 
    $$c^0_{r_1 \cdots r_s} = \left\langle \hat{\vartheta}_{r_1},\dots, \hat{\vartheta}_{r_s}\right\rangle.$$
\end{lemma}

\begin{proof}
  Taking the product of theta functions gives
  $$\sum_{(\gamma_1,\ldots,\gamma_s)} c(\gamma_1) \hat{z}^{\varv(\gamma_1)} \cdots c(\gamma_s)\hat{z}^{\varv(\gamma_s)},$$
  where the sum is over tuples $(\gamma_1,\dots,\gamma_s)$  of broken lines such that $\gamma_i(0) = Q$ for some fixed $Q$ (near the origin). Taking the coefficient of $\hat{\vartheta}_0$ equal to $1$ ensures that $\varv(\gamma_1) + \dots + \varv(\gamma_s) = 0$ so that we have 
  $$c^0_{r_1 \cdots r_s} = \sum_{(\gamma_1,\ldots,\gamma_s)} c(\gamma_1)  \cdots c(\gamma_s)\hat{z}^{-\sum_{i=2}^{s}\varv(\gamma_i)}\hat{z}^{\varv(\gamma_2)}\cdots\hat{z}^{\varv(\gamma_s)},$$
  and the result now follows from the product rule $\hat{z}^{r} \cdot \hat{z}^{r'} = q^{\frac{1}{2}r \wedge r'}\hat{z}^{r+r'}$.
\end{proof}
 
\begin{remark}
The quantum mirror of \cite{bousseau2020quantum} is defined with respect to a different quantum scattering diagram, analogous to the canonical scattering diagram of \cite{gross2015mirror}, but the output is essentially the same, and we explain this briefly. The algebra is defined by fixing the data $(B,\Sigma)$, $P$, $J$, and $\varphi$. Here $(B,\Sigma)$ is an integral affine manifold (with singularities) defined via the geometry of $(Y_m,\partial Y_m)$, $P$ is a toric monoid containing $\mathrm{NE}(Y_m)$, $\varphi$ is a $P_{\RR}^{\mathrm{gp}}$-valued PL function on $B$, and $J$ is a monomial ideal of $P$. The product structure is also defined with respect to the scattering diagram by counts of broken lines. By \cite[Proposition 4.10]{bousseau2020quantum}, the canonical scattering diagram is related to the scattering diagram of Definition~\ref{Definition : scattering completion}, by a piecewise-linear map. As a result, broken lines go to broken lines, and the difference in their definition of the product of \cite[Theorem 3.11]{bousseau2020quantum} cancels with the presence of the PL function to recover the same product structure.
\end{remark}

\subsection{Quantum scattering to $\boldsymbol{q}$-refined tropical curves}

Let
$$\left[\left\langle \hat{\vartheta}_{r_1}, \dots, \hat{\vartheta}_{r_n} \right\rangle\right]^{\mathsf{sym}} = \frac{1}{(s-1)!}\sum_{\omega \in \Omega_s}\left\langle \hat{\vartheta}_{r_{\omega(1)}},\ldots,\hat{\vartheta}_{r_{\omega(s)}} \right\rangle.$$

\begin{lemma}
    In the factored scattering diagram, there is an equality 
\begin{equation*} 
  \left[\left\langle \hat{\vartheta}_{r_1}, \dots, \hat{\vartheta}_{r_n} \right\rangle\right]^{\mathsf{sym}} = \sum_{p} \sum_{k \vdash p} \sum_{A_{j \ell}} \left( \sum_{\Gamma \in T_{A_{j \ell}}} \prod_{V \in V(\Gamma)}m_{V}\left(q^{\frac{1}{2}}\right)\right) \prod_{j=1}^n \prod_{\ell \geq 1} \left(\frac{(-1)^{\ell -1}}{\ell} \frac{q^{\frac{1}{2}} - q^{-\frac{1}{2}}}{q^{\frac{\ell}{2}} - q^{-\frac{\ell}{2}}}\right)^{k_{\ell j}} \left(\ell!^{k_{\ell j}}\right) \left(\prod_{A \in A_{j \ell}} \prod_{a \in A} u_{ja}\right).
\end{equation*}
The first sum is over all $p=(p_1,\dots,p_n)$ such that $\sum_{i=1}^n p_i m_i = \sum_{j=1}^s r_j$. Here, $k \vdash p$ denotes a partition  $(k_{\ell j})_{\ell \geq 1}$ of $p_j$ for each $j=1,\dots,n$, $A_{j \ell}$ is a set of $k_{\ell j}$ disjoint subsets of\, $\{1,\dots,N\}$ of size $\ell$, and $T_{A_{j \ell}}$ is the set of genus zero tropical curves $\Gamma$ having an unbounded edge of asymptotic direction $r_i$ for $i=1,\dots, s$. Furthermore, for every $j=1,\dots,n$ and $\ell \geq 1$, $A \in A_{j \ell}$ is an unbounded edge of weight $\ell$ coinciding with $\fd_{j \ell A}$, with an $s$-valent point condition at a point $Q \in \RR^2$.
\end{lemma}

\begin{proof}
    We have that 
\begin{equation}\label{eq : coefficient scatter}
    \left\langle \vartheta_{r_1}, \dots, \vartheta_{r_n} \right\rangle =\sum_{(\gamma_1,\dots,\gamma_s)} c(\gamma_1) \cdots  c(\gamma_s) \prod_{2 \leq i < j \leq s} q^{\frac{1}{2}\varv(\gamma_i) \wedge \varv(\gamma_j)}, 
\end{equation}
 where the sum is taken over tuples of quantum broken lines where $\gamma_i$ has start direction $r_i$ and ends at $Q$, and such that the broken lines collectively are balanced at $Q$. For each $i=1,\dots, s$, the broken line $\gamma_i$ bends at various rays $\fd$ in the factored scattering diagram. Recall that by Lemma~\ref{Lemma : ray bijection} there is a bijection between rays $\fd$ of the factored scattering diagram and tropical curves such that the Hamiltonian associated to the ray is given by \eqref{eq : hamiltonian expression}.

 If $\hat{z}^k$ is the monomial attached to a segment of the broken line $\gamma_i$, the wall-crossing automorphism of $\fd$ acts by 
 $$\hat{z}^k \longmapsto \exp\left(\hat{H_{\fd}}\right)\hat{z}^k \exp\left(-\hat{H}_{\fd}\right).$$
 Since the variables $u_{ij}^2$ are equal to $0$, this implies that if $\gamma_i$ bends when it hits $\fd$ at $V$, the outgoing Hamiltonian picks up a factor of
 $$m_V\left(q^{\frac{1}{2}}\right) \cdot m_{\Gamma}\left(q^{\frac{1}{2}}\right)\prod_{j=1}^n \prod_{\ell \geq 1}\left(\frac{(-1)^{\ell}}{\ell} \frac{q^{\frac{1}{2}} - q^{-\frac{1}{2}}}{q^{\frac{\ell}{2}} - q^{-\frac{\ell}{2}}}\right)^{k_{\ell,j}} (\ell !)^{k_{\ell,j}} \left(\prod_{A \in A^{\Gamma}_{j \ell}} \prod_{a \in A} u_{ja}\right),$$
    where the edges are considered to be the incoming and outgoing edges of the broken line together with the part of $\fd$ which has finite length. Adjoining this part of the tropical curve to $\gamma_i$, one can associate to the tuple $(\gamma_1,\dots,\gamma_n)$ a tropical curve with an $s$-valent vertex. 

    On the other hand, every tropical curve with an $s$-valent vertex at $Q$, with asymptotic directions given by $r_1,\dots,r_s$ as well as some subset of the $\fd_{j \ell A}$, gives rise to broken lines $(\gamma_1,\dots,\gamma_s)$ appearing in the right-hand side of \eqref{eq : coefficient scatter}. The result now follows from the fact that the $q$-refined multiplicity of the $s$-valent vertex as in Definition~\ref{Definition : blechman-shustin} appears from the factor $$ \prod_{2 \leq i < j \leq s} q^{\frac{1}{2}\varv(\gamma_i) \wedge \varv(\gamma_j)}$$ after averaging over cyclic permutations.
\end{proof}

\begin{corollary}\label{corollary : scatter to tropical}
    There is an equality 
    $$ \left[\left\langle\hat{\vartheta}_{r_1},\ldots,\hat{\vartheta}_{r_s}\right\rangle\right]^\mathsf{sym}
    = \sum_p \sum_{k \vdash p} \mathsf{N}^{\Delta,s}_{\mathsf{trop, \Delta_{k}}}(q)\left( \prod_{j=1}^n \prod_{\ell \geq 1} \frac{1}{k_{\ell j}!} \left(\frac{(-1)^{\ell -1}}{\ell} \frac{q^{\frac{1}{2}} - q^{-\frac{1}{2}}}{q^{\frac{\ell}{2}} - q^{-\frac{\ell}{2}}}\right)^{k_{\ell j}}\right)\left(\prod_{j=1}^n t_j^{p_j}\right).$$
    The first sum is over all $p=(p_1,\dots,p_n)$ such that $\sum_{i=1}^n p_i m_i = \sum_{j=1}^s r_j$. Here, $k \vdash p$ denotes a partition $(k_{\ell j})_{\ell \geq 1}$ of $p_j$ for each $j=1,\dots,n$, $\Delta$ consists of\, $k_{\ell j}$ copies of the vector $-\ell m_j$ as well as $r_1, \dots, r_s$, 
    and $\Delta_k$ denotes $\Delta \setminus \{r_1,\dots,r_s\}$.
\end{corollary}

\begin{proof}
This follows the same argument as in~\cite[Proof of Proposition 6.2]{bousseau2020quantumtrop}. Let $B=\cup_{A \in A_{j \ell}} A$ be a subset of $\{1,\dots,N\}$ of size $\sum_{\ell \geq 1} \ell k_{\ell j} = p_j$. Conversely, the number of ways to write a set
$B$ of $p_j = \sum_{\ell \geq 1} \ell k_{\ell j}$ elements as a disjoint union of subsets, $k_{\ell j}$ of them being of size $\ell$, is equal to $$\frac{p_j!}{\prod_{\ell \geq 1} k_{\ell j}! (\ell !)^{k_{\ell j}}}.$$ Replacing the sum over $A_{j\ell}$ by a sum over $B$ and using the fact that $$t_j^{p_j} = \sum_{B \subset \{1,\dots, N\}, |B| = p_j} p_j! \prod_{b \in B} u_{jb}$$ gives the result after identifying
\begin{equation*}\pushQED{\qed}
  \mathsf{N}^{\Delta,s}_{\mathsf{trop, \Delta_{k}}}(q) = \sum_{\Gamma \in T_{A_{j \ell}}} \prod_{V \in V(\Gamma)}m_{V}\left(q^{\frac{1}{2}}\right).\qedhere \popQED
  \end{equation*}
\renewcommand{\qed}{}    
\end{proof}

\section{From logarithmic Gromov--Witten invariants to tropical curve counts}

In this section, we use the degeneration formula to relate the logarithmic Gromov--Witten invariants of $Y_m$ to toric logarithmic Gromov--Witten invariants. The latter can be connected with refined tropical curve counting due to Theorem~\ref{thm : bdyrefined}. This section follows~\cite[Section 5]{bousseau2020quantum}.

\subsection{Degeneration}

First we recall the degeneration of~\cite[Section 5.3]{gross2010vertex} (see also \cite{bousseau2020quantumtrop,mandel2021theta}). Recall that the tuple $m= (m_1, \dots,m_n)$ of non-zero vectors in $\ZZ^2$ fixes a toric surface $\overline{Y}_m$. For every $j=1,\dots,n$, let $x_j$ denote a general point on the interior of the boundary divisor $D_{m_j}$. Take the total space of the deformation to the normal cone of $D_{m_1} \cup \dots \cup D_{m_n}$ in $\overline{Y}_m$ (we can assume the $D_{m_j}$ are disjoint). Then blow up the strict transforms of the constant sections determined by $x_j$ to obtain a family
$$\epsilon \colon \cY_m \longrightarrow \AAA^1.$$
The general fibre is $Y_m$, and the central fibre is $$\overline{Y}_m \cup \bigcup_{j=1}^n \tilde{\PP}_{j},$$ where $\tilde{\PP}_j$ is the blow-up of the projective completion of the normal bundle of $D_j$ in $\overline{Y}_m$ at the point corresponding to $x_j$. Equipping $\cY_m$ with the divisorial logarithmic structure with respect to the central fibre together with the strict transform of $\partial \overline{Y}_m \times \AAA^1$ and equipping $\AAA^1$ with the divisorial logarithmic structure with respect to the origin makes $\epsilon$ into a logarithmically smooth morphism. Restricting to the central fibre, there is a logarithmically smooth morphism $\cY_{m,0} \rightarrow \mathsf{pt}_{\NN}$. Write $\mathsf{M}_{g,\vartheta}(\cY_{m,0},\beta_p)$ for the moduli space of genus~$g$, $(s+1)$-pointed stable logarithmic maps to $\cY_{m,0}$ over $\mathrm{pt}_{\NN}$, of class $\beta_p$, with tangency given by $\vartheta$.

By the deformation invariance of logarithmic Gromov--Witten invariants, we have that $$\mathsf{N}_{g,\vartheta}^{\beta_p} = \int_{\left[\mathsf{M}_{g,\vartheta}(\mathcal{Y}_{m,0},\beta_p)\right]^\mathsf{vir}}(-1)^g\lambda_g \ev^* (\mathrm{pt}) \psi^{s-2}.$$

\subsection{Decomposition}

Tropical data associated to the morphism $\cY_{m,0} \rightarrow \mathsf{pt}_{\NN}$ may be recorded as a polyhedral complex $\cY^{\mathsf{trop}}_{m,0}$ with metrised edge lengths; see~\cite[Section 5.3]{bousseau2020quantumtrop}. Our polyhedral complex is obtained by subdividing the fan of $\overline{Y}_m$, whose vertex is called $v_0$, by placing a vertex $v_j$ at distance 1 along each ray in direction $m_j$ and adding unbounded edges attached to $v_j$ parallel to those rays adjacent to $-\RR_{\geq 0}m_{j}$.

\begin{figure}
    \centering
    \includegraphics[width=10cm]{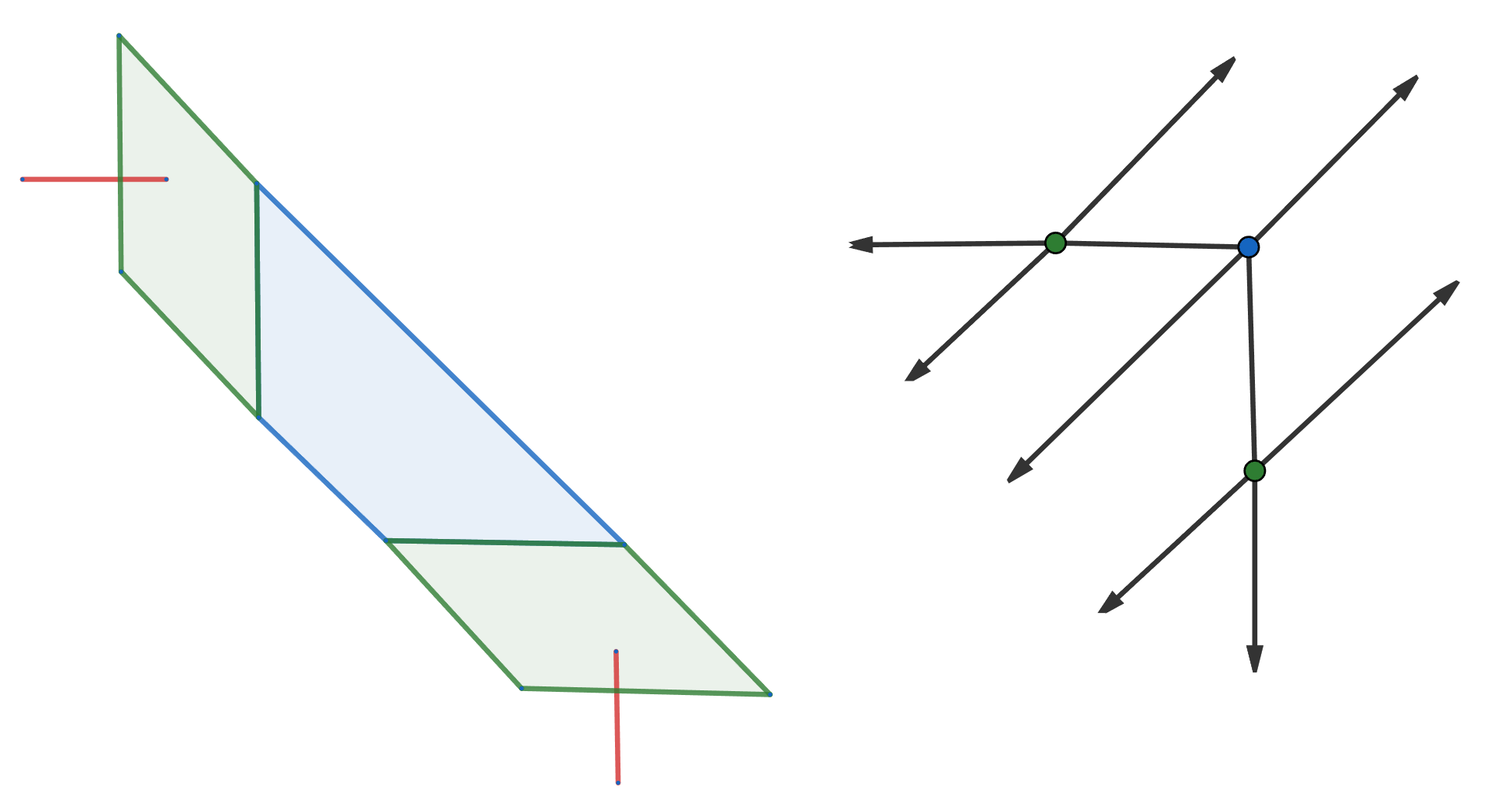}
    \caption{An example of the central fibre $\cY_{m,0}$ and the tropicalisation $\cY_{m,0}^{\mathsf{trop}}$. }
    \label{Figure : Target example}
\end{figure}

The decomposition formula of \cite{abramovich2020decomposition} allows us to write $N_{g,\vartheta}^{\beta_p}$ as a sum over rigid tropical curves $h \colon \Gamma \rightarrow \cY^{\mathsf{trop}}_{m,0}$.

The moduli space of $(s+1)$-pointed, genus $g$ stable logarithmic maps marked by $h$, denoted by $\mathsf{M}^{h}_{g,\vartheta}$, is a proper Deligne--Mumford stack with a natural perfect obstruction theory and a forgetful morphism
$$i_{h} \colon \mathsf{M}^{h}_{g,\vartheta} \longrightarrow \mathsf{M}_{g,\vartheta}(Y_{m,0},\beta_p).$$
The decomposition formula \cite{abramovich2020decomposition} tells us that the following holds. 

\begin{proposition}
  We have that 
  $$\mathsf{N}^{\beta_p}_{g,\vartheta} = \sum_h \frac{n_h}{|\mathsf{Aut}(h)|} \mathsf{N}_h,$$
  where
  $$\mathsf{N}_h = \int_{\left[\mathsf{M}_{g,\vartheta}^h\right]^\mathsf{vir}}(-1)^g\lambda_g \ev^* (\mathrm{pt}) \psi^{s-2}.$$
\end{proposition}

In this situation, only certain rigid tropical types will have a non-vanishing contribution in the decomposition formula. In order to index these, we let $k=(k_1,\dots,k_n)$ be a partition of $p=(p_1,\dots,p_n)$, \textit{i.e.} $k_j$ equals $(k_{\ell j})_{\ell \geq 1}$, finitely many non-zero integers with $$\sum_{\ell \geq 1} \ell k_{\ell j} = p_{j}.$$
We write $k \vdash p$. Moreover, let $s(k) = \sum_{j=1}^n \sum_{\ell \geq 1} k_{\ell j}$ and $w(k) = (w_{1}(k), \dots, w_{s(k)}(k))$, where the $w_{j}(k) \in \ZZ^2$ for each $j$ are such that $w(k)$ contains $k_{\ell j}$ copies of the vector $-\ell m_{j}$. 

Recall that the vectors, $\vartheta_1,\dots,\vartheta_s$ are multiples of the rays of the fan of $\overline{Y}_m$. Suppose $\vartheta_{i_1},\dots,\vartheta_{i_b}$ are multiples of $-m_{a_1},\dots,-m_{a_b}$. 

We now define a class of rigid tropical curves 
   \begin{equation}\label{defn:GraphFromPartition}
        h_{k,\underline{g}}^s \colon \Gamma_{k,\underline{g}}^s \longrightarrow \cY_{m,0}^{\mathsf{trop}}       
   \end{equation} 
 depending on a partition $k \vdash p$ as well as a tuple $\underline{g} = (g_0,\dots,g_{s(k)})$  of non-negative integers with $\sum_{i=0}^{s(k)} g_i = g$.

\begin{definition}
 Let $\Gamma_{k,\underline{g}}^s$ be the genus zero graph with
   \begin{itemize}
       \item $V(\Gamma_{k,\underline{g}}^s) = \{V_0,\dots,V_{s(k)},V'_{1},\dots,V'_{b}\}$,
       \item $E_{f}(\Gamma_{k,\underline{g}}^s) = \{E_1,\dots,E_{s(k)},E'_1,\dots,E'_b\}$ with $V_0,V_j \in E_j$  and $V_0,V'_j \in E'_j$,
       \item $E_{\infty}(\Gamma_{k,\underline{g}}^s) = \{F_1,\dots,F_s,F\}$ with $V'_j \in F_{i_{j}}$ for $j=1,\dots,b$, $V_j \in F_j $ otherwise, and $V_0 \in F$.
\end{itemize}
To the graph $\Gamma^{s}_{k,\underline{g}}$ assign
\begin{itemize}
       \item the genus $g_j$ to $V_j$ for $j=0,\dots,s(k)$ and genus zero to $V'_j$ for $j=1,\dots,b$;
       \item the length $ \ell(E_j) = \frac{1}{|w_{j}(k)|}$ for $j=1,\dots,s(k)$ and the length $\ell(E'_j) = \frac{1}{|\vartheta_{i_j}|}$ for $j=1,\dots,b$.
\end{itemize}
We define a tropical curve $h_{k,\underline{g}}^s \colon \Gamma^s_{k,\underline{g}} \rightarrow \cY^{\mathsf{trop}}_{m,0}$ as follows: 
\begin{itemize}
    \item Set $h(V_0) = v_0$ and $h(V_j) = v_i$ if $w_{j} = -\ell m_i$, for $j=1,\dots,s(k)$.
    \item Set $h(V'_j) = v_{a_{j}}$ and $h(F_{i_j}) = v_{a_{j}} + \RR_{\geq 0} \vartheta_{i_j}$ for $j=1,\dots, b$, and $h(F_j) = \vartheta_j$ otherwise.
    \item Set $h(F)= v_0$.
    \item Decorate $V_0$ with the curve class $\beta_{w(k)} \in H_{2}(\overline{Y}_m,\ZZ)$ determined by the contact data $w(k)$.
    \item If $w_{j} = -\ell m_i$, decorate $V_{j}$ with the curve class $\ell [C_i]$, where $[C_i]$ corresponds to the strict transform of a fibre class passing through the point corresponding to $x_i$ in $\tilde{\PP}_i$, for $j=1,\dots,s(k)$.
    \item Decorate $V'_j$ with the curve class $|\vartheta_{i_j}|[C'_{i_j}]$, where $[C'_{i_j}]$ corresponds to the strict transform of a fibre class in $\tilde{\PP}_{i_j}$, for $j=1,\dots,b$.
\end{itemize} 
\end{definition}

\begin{figure}[H]\hspace{-6em}\vspace{-1em}
    \centering
    \includegraphics[width=9cm]{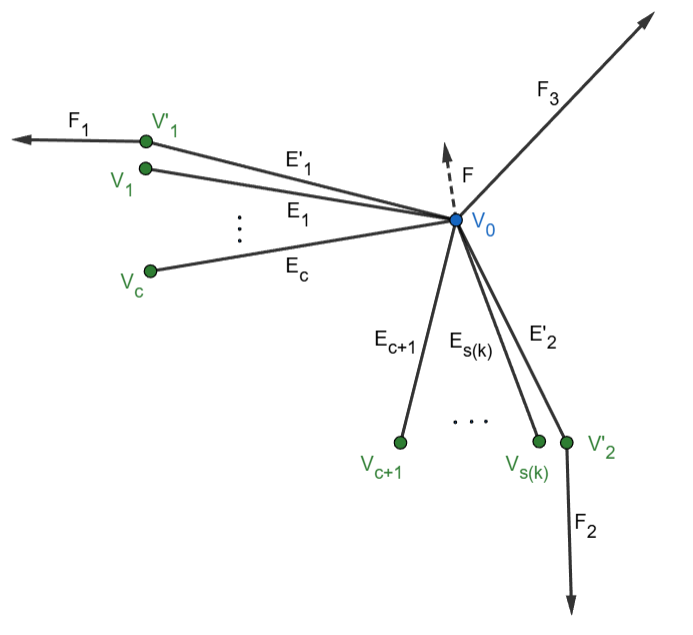}
    \caption{Example of $\Gamma_{k,\underline{g}}^s$ for target as in Figure~\ref{Figure : Target example}.}
    \label{Figure : tropical type example}
\end{figure}

We claim the only rigid tropical types appearing with non-vanishing contribution in the decomposition formula are the $h_{k,\underline{g}}^s$.

\begin{proposition}\label{proposition : rigid types}
    Every rigid tropical type $h$ with a non-zero contribution to the decomposition formula is of the form $h= h^s_{k,\underline{g}}$.
\end{proposition}

The proof will use the following lemmas. We will say that an edge $E$ of a tropical curve is \textit{on a zero line} if the line containing $E$ passes through the origin.

\begin{lemma}\label{lem:ContinuePaths}
    Suppose $h\colon \Gamma \rightarrow \cY_{m,0}^{\mathsf{trop}}$ is a rigid tropical type. Fix an edge $\gamma_0$ and a vertex $V$ of $\gamma_0$. Assume $\gamma_0$ is not on a zero line. Then there is a sequence of edges $\gamma_0,\ldots,\gamma_n$ with the following properties:
    \begin{enumerate}
        \item\label{l:CP-1} The edges $\gamma_i$, $\gamma_{i+1}$ meet at a vertex, and $\gamma_0$, $\gamma_1$ meet at $V$.
        \item\label{l:CP-2} Either $\gamma_n$ is unbounded, or $\gamma_n = \gamma_0$.
    \end{enumerate}
\end{lemma}

\begin{proof}
    We first show that fixing an edge $\gamma_0$ not on a zero line and a vertex $v$ of $\gamma_0$, there exists an edge $\gamma_1$ which meets $\gamma_0$ at $v$ and is not on a zero line. Indeed, we can see this is true by applying the balancing condition at $v$ in the direction $\mathbf{e}$ orthogonal to the line joining $v$ to the origin. Since $\gamma_0$ is not orthogonal to~$\mathbf{e}$, the vertex $v$ must be adjacent to some other edge $\gamma_1$ whose dot product with $\mathbf{e}$ is non-zero.

    We now recursively construct our sequence of edges. By the previous paragraph, given a sequence $\gamma_1,\ldots,\gamma_i$ of edges with property~\eqref{l:CP-1} and such that $\gamma_i$ is bounded, there exists an edge $\gamma_{i+1}$ which does not lie on a zero line and which meets $\gamma_{i}$ at the vertex not adjacent to $\gamma_{i-1}$. Tropical curves have finitely many edges, and the process must terminate, which can only occur when we find some unbounded $\gamma_{n}$, or else we have constructed a cycle.
\end{proof}

We define a \textit{continuable} edge to be an edge which is either an unbounded edge or a bounded edge connected to a 2-valent vertex whose other edge is unbounded.

\begin{lemma}\label{lem:continuable edges}
    Suppose a rigid type $h\colon \Gamma \rightarrow \cY_{m,0}^{\mathsf{trop}}$ makes a non-zero contribution to the decomposition formula. Then $\Gamma$ contains a vertex $V_0$ with the following properties:
    \begin{enumerate}
        \item $V_0$ is mapped to the origin.
        \item $V_0$ has $s$ continuable edges.
    \end{enumerate} 
\end{lemma}

\begin{proof}
There is a distinguished vertex $V_0 \in V(\Gamma)$ containing the non-tangency marking. We can assume that $h(V_0) = v_0$, for example by imposing the point constraint in the interior of the toric component. By \cite[Remark 6.1.1]{ranganthan2022logarithmic}, after potentially subdividing the target further, we can assume that $h$ maps vertices to vertices and edges to edges. This allows us to make the gluing argument, as in Proposition~\ref{Proposition : gluing}, for any $h$. Consequently, there is a virtual dimension constraint on each vertex, which enforces restrictions on $h$ as well as on which term in the decomposition of the diagonal contributes to the insertions.

Next we analyse the virtual dimension constraint for $V_0$. The virtual dimension of the associated moduli space is $g_0 + k + k'$, where $k+k'$ is the valency of $V_0$, where $k$ is the number of bounded edges adjacent to $V_0$ for which the contribution of the diagonal term consists of $\mathsf{pt}_E$ for $V_0$. We have that the dimension of the insertions is $g_0 + s +k$, so $k'= s$.
 
We will argue that all such edges $k'$ are continuable. Indeed, if an edge $\gamma_{-1}$ is not continuable, then the graph must fork at the vertex of $\gamma_{-1}$ which is not the origin. By Lemma~\ref{lem:ContinuePaths}, each branch of the forking can be continued until we find an unbounded edge. Since $\Gamma$ is of genus zero, each of the unbounded edges found in this way is unique, and it follows that $\Gamma$ has more than $s$ unbounded edges. 
\end{proof}

We can also exclude the rigid types where a vertex $V \neq V_0$ attached to a continuable bounded edge has higher genus by \cite[Lemma 14]{bousseau2018tropical}.

\begin{proof}[Proof of Proposition~\ref{proposition : rigid types}]
  By Lemma~\ref{lem:continuable edges}, we are left to show that the $k$ bounded edges adjacent to $V_0$ for which the diagonal contributes a $\mathsf{pt}_E$ to the insertions for $V_0$ each attach to a vertex $V_+$ which is one valent with image $v_j$ for $j \in \{1,\dots, n\}$. If $V_+$ is such a vertex, then the diagonal must contribute $1$ to the invariant for $V_+$ and $\mathsf{pt}_E$ to the invariant for $V_0$, because the virtual dimension of the moduli space associated to $V_+$ is  $g(V_+)$. Conversely, suppose $V_+$ is not a one-valent vertex. By the argument of \cite[Lemma 5.8, penultimate paragraph]{bousseau2020quantumtrop}, there is no additional edge connecting $V_+$ to a vertex with image $v_0$, since there are no additional insertions to make the contribution to that vertex non-zero. Finally, if there were an edge connecting $V_+$ to another vertex whose image is not $V_j$, then by Lemma~\ref{lem:ContinuePaths} there would be a sequence of edges which either ends in a cycle, which would give zero contribution due to the $\lambda_g$ insertion, or else ends in an unbounded edge. But we have already accounted for $s+1$ unbounded edges of $\Gamma$.    
\end{proof}

Let $k_{\ell j a}$ denote the number of vertices of $\Gamma_{k,\underline{g}}^s$ having genus $a$ amongst the $k_{\ell j}$ vertices with curve class $\ell [C_j]$ so that $$\sum_{a \geq 0} k_{\ell j a } = k_{\ell j}.$$

\begin{corollary}\label{prop : decomposition}
  We have that 
    $$\mathsf{N}_{g,\vartheta}^{\beta_p} = \sum_{k \vdash p} \sum_{\underline{g}} \frac{n_{h^s_{k,\underline{g}}}}{\left|\mathsf{Aut}\left(h^s_{k,\underline{g}}\right)\right|} \mathsf{N}_{h^s_{k,\underline{g}}} =  \sum_{k \vdash p} \sum_{\underline{g}} \mathsf{lcm}\left(|w_{i}(k)|, \left|\vartheta_{i_j}\right|\right) \left(\prod_{j=1}^n \prod_{\ell \geq 1} \prod_{a \geq 0} \frac{1}{k_{\ell j a}!}\right)\mathsf{N}_{h^s_{k,\underline{g}}}.$$
\end{corollary}

\begin{proof}
  By Proposition~\ref{proposition : rigid types}, all curves that contribute in the decomposition formula arise from the construction of \eqref{defn:GraphFromPartition}. The fact that $n_{h^s_{k,\underline{g}}} = \mathsf{lcm}(|w_{i}(k)|, |\vartheta_{j}|)$ follows from the fact that $n_{h^s_{k,\underline{g}}}$ is the smallest integer such that after scaling by $n_{h^s_{k,\underline{g}}}$, $h^s_{k,\underline{g}}$ gets integral vertices and integral lengths. The automorphism factor can be computed by observing that those one-valent vertices which have the same curve class and genus can be permuted.
\end{proof}

It now suffices  to compute $\mathsf{N}_{h_{k,\underline{g}}}^s$.

\subsection{Vertex contributions}
Consider $h_{k,\underline{g}}^s\colon \Gamma_{k,\underline{g}}^s \rightarrow \cY_{m,0}^{\mathsf{trop}}$. There are three types of vertices of $\Gamma_{k,\underline{g}}^s$: 
\begin{enumerate}
    \item the vertex $V_0$ which maps via $h_{k,\underline{g}}^s$ to $v_0$, corresponding to $\overline{Y}_m$; 
    \item the vertices $V_{j}$ for $j=1,\dots,s(k)$, which map to vertices corresponding to $\tilde{\PP}_i$; 
    \item\label{type3} the vertices $V'_{j}$ for $j=1,\dots, b$, which map to vertices corresponding to $\tilde{\PP}_i$.
\end{enumerate}
\subsubsection{Toric contribution}

Consider the multiset $\Delta$ consisting of all $v_{V_0,E}$ for edges adjacent to $V_0$. This multiset consists of $\vartheta_{1},\dots,\vartheta_{s}$ together with the elements of $w(k)= (w_1(k),\dots,w_{s(k)}(k))$ and one copy of the zero vector. Associated to this multiset, we have the moduli space $\mathsf{M}_{g_0,\Delta}$  of genus $g_0$, $(s(k) + s + 1)$-pointed stable logarithmic maps to $\overline{Y}_m$. If $w_{j}(k) = -\ell m_i$, let
$$\ev_{j} \colon \mathsf{M}_{g_0,\Delta} \longrightarrow D_{i}$$
denote the evaluation morphism at the $\supth{j}$ marking for $j=1,\dots,s(k)$. Let
$$\ev \colon \mathsf{M}_{g_0,\Delta, \Delta_k} \longrightarrow \overline{Y}_m$$
denote the evaluation morphism at the last marked point. Define the \emph{toric contribution} to be $$\mathsf{N}_{g_0,\Delta, \Delta_k} = \int_{[\mathsf{M}_{g_0,\Delta}]^{\vir}} (-1)^{g_0} \lambda_{g_0} \prod_{j=1}^{s(k)} \ev_{j}^{*}(\mathsf{pt}_{D_i})\ev^* (\mathsf{pt})\psi^{s-2}.$$

\subsubsection{Non-toric contribution}
For each vertex $V_j$ for $j=1,\dots,s(k)$, there is an adjacent edge $E_j$ with $v_{V_j,E_j} = \ell m_i$ for some $i$. Consider the moduli space $\mathsf{M}_{a,\ell}(\tilde{\PP}_i)$  of one-pointed, genus $a$ stable logarithmic maps to $\tilde{\PP}_i$ of degree $\ell[C_i]$ with contact order $\ell$ along $D_{m_i}$. Define the \emph{non-toric contribution} to be
$$\mathsf{N}_{a}^{\ell \tilde{\PP}_i} = \int_{[\mathsf{M}_{a,\ell}(\tilde{\PP}_i)]^{\vir}}(-1)^a \lambda_{a}.$$
We will now express $\mathsf{N}_{h_{k,\underline{g}}^s}$ as a product over the \emph{toric} and \emph{non-toric} contributions. \textit{A priori}, there are also vertices of type~\eqref{type3}, but their contributions will cancel with an overall factor. In~\cite[Section 5.5]{bousseau2020quantumtrop}, the author proves a gluing statement at the level of virtual classes on the locus where the curve does not map to the codimension two strata of the pieces of the central fibre. This is sufficient for~\cite[Proposition 5.6]{bousseau2020quantumtrop} as the contribution from the complement of this locus will be killed by the presence of the $\lambda_g$ insertion. As in~\cite[Section 3.4]{kennedyhunt2023tropical}, this is no longer sufficient for the purposes here, the key point being that~\cite[Lemma 5.7]{bousseau2020quantumtrop} is no longer true in our situation. Instead we follow the same strategy as in \cite[Section~3.4]{kennedyhunt2023tropical}, by using the gluing formula of \cite{ranganthan2022logarithmic}. The upshot is the following.

\begin{proposition}\label{Proposition : gluing}
  We have that
    $$\mathsf{N}_{h_{k,\underline{g}}^s} = \frac{\mathsf{N}_{g_0,\Delta, \Delta_k}}{\mathsf{lcm}\left(|w_i(k)|,\left|\vartheta_{i_j}\right|\right)}\left(\prod_{j=1}^n \prod_{\ell \geq 1} \ell^{k_{\ell j}}\prod_{a \geq 0} \left(\mathsf{N}^{\ell \tilde{\PP}_i}_{a}\right)^{k_{\ell j a}}\right).$$
\end{proposition}

\begin{proof}
We can induct on the bounded edges of $\cY^{\mathsf{trop}}_{m,0}$, and so we can assume that there is only one edge. For simplicity, also assume $b=0$. 

We wish to turn an integral over $[\mathsf{M}_{h}]^{\vir}$ into one over
$$\kappa_{h}^!\left(\left[\mathsf{M}_{g_0,\Delta}\right]^{\vir} \times \prod_{j=1}^{n}\prod\limits_{\substack{a \geq 0 \\ \ell\geq 1}}  \left(\left[\mathsf{M}_{a,\ell}(\tilde{\PP})\right]^{\vir}\right)^{k_{\ell j a}}\right)$$ in the fibre product below, where $\kappa_h$ is the diagonal: 
\begin{equation*}
    \begin{tikzcd}
      \times \mathsf{M} \arrow[d] \arrow[r] & {\mathsf{M}_{g_0,\Delta} \times \prod_{j,\ell, a}  \left(\mathsf{M}_{a,\ell}\left(\tilde{\PP}\right)\right)^{k_{\ell j a}}} \arrow[d,"\ev_D"] \\
D^{s(k)} \arrow[r, "\kappa_h"]            & (D \times D)^{s(k)}\rlap{.}                                                                
\end{tikzcd}
\end{equation*}
If one of the arrows in the above diagram is combinatorially flat, this will follow from~\cite[Section 12.3]{maulik2025logarithmicenumerativegeometrycurves}, as the fine and saturated fibre product coincides with the ordinary fibre product. Here we can rewrite the diagram as 
\begin{equation*}
    \begin{tikzcd}
\times \mathsf{M} \arrow[d] \arrow[r] & {\prod_{j,\ell,a}  \left(\mathsf{M}_{a,\ell}\left(\tilde{\PP}\right)\right)^{k_{\ell j a}}} \arrow[d] \\
\mathsf{M}_{g_0,\Delta} \arrow[r]            & D^{s(k)}             
\end{tikzcd}
\end{equation*}
so that the right-hand vertical arrow is combinatorially flat, as it is a product of $s(k)$ arrows, each com\-bi\-na\-to\-ri\-ally flat. Consequently, we have that $\mathsf{N}_{h_{k,\underline{g}}^s}$ is equal to
$$\frac{\prod_{j=1}^{n}\prod_{\ell \geq 1}\ell^{k_{\ell j}}}{\mathsf{lcm}(|w_i(k)|)}\int_{[\mathsf{M}_{g_0,\Delta}]^{\vir} \times \prod_{j,\ell, a}  ([\mathsf{M}_{a,\ell}(\tilde{\PP})]^{\vir})^{k_{\ell j a}}} \left((-1)^{g_0} \lambda_{g_0} \times \prod_{j,\ell,a} (-1)^{a} \lambda_{a} \times \dots \times (-1)^{a} \lambda_{a} \right)\ev_D^*(\kappa_h) \ev^*(\pt)\psi^{s-2}.$$
We can decompose the diagonal as $\kappa_h = \prod_{i=1}^{s(k)}\pt_{D} \otimes 1 + 1 \otimes \pt_{D}$. Due to the virtual dimension of the moduli spaces, the only summand which contributes is $\pt_{D} \otimes 1$, and so we have $$\mathsf{N}_{h_{k,\underline{g}}^s} = \frac{\prod_{j=1}^{n}\prod_{\ell \geq 1}\ell^{k_{\ell j}}}{\mathsf{lcm}(|w_i(k)|)}\left(\int_{[\mathsf{M}_{g_0,\Delta}]^{\vir}} (-1)^{g_0} \lambda_{g_0} \prod_{j=1}^{s(k)} \ev_{j}^{*}(\mathsf{pt}_{D})\ev^* (\mathsf{pt})\psi^{s-2}\right) \left(\prod_{a \geq 0} \int_{[\mathsf{M}_{a,\ell}(\tilde{\PP}_i)]^{\vir}}(-1)^a \lambda_{a}\right),$$ which gives the desired formula. In the case where $b\neq 0$, the argument is the same but the contribution from each vertex $V_i$ will cancel with the weight of the edge; see~\cite[Lemma 15]{bousseau2018tropical}.
\end{proof}

Combining the above with Proposition~\ref{prop : decomposition} gives us the following corollary.

\begin{corollary}
We have that 
  \begin{equation*}
        \mathsf{N}_{g,\vartheta}^{\beta_p} = \sum_{k \vdash p} \sum_{\underline{g}} \mathsf{N}_{g_0,\Delta, \Delta_k} \left(\prod_{j \geq 1} \prod_{\ell \geq 1} \ell^{k_{\ell j}}\prod_{a \geq 0} \frac{1}{k_{\ell j a}!}\left(\mathsf{N}_{a}^{\ell \tilde{\PP}_j}\right)^{k_{\ell j a}}\right). 
    \end{equation*}
\end{corollary}

\begin{proposition}\label{Proposition : GW to toric}
  We have that 
    \begin{equation}
        \sum_{g \geq 0 } \mathsf{N}_{g,\vartheta}^{\beta_p} u^{2g} = \sum_{k \, \vdash \, p} \left(\sum_{g \geq 0} \mathsf{N}_{g,\Delta, \Delta_k}u^{2g + s(k)} \right)\prod_{j=1}^n \prod_{\ell \geq 1} \frac{\ell^{k_{\ell j}}}{k_{\ell j}!} \left(\frac{(-1)^{\ell - 1}}{\ell} \frac{1}{2 \sin\left(\frac{\ell u}{2}\right)}\right)^{k_{\ell j}}.
    \end{equation}
\end{proposition}

\begin{proof}
  Let $F^{\ell\tilde{\PP}_j}(u) = \sum_{a \geq 0} \mathsf{N}^{\ell \tilde{\PP}_j}_a u^{2a - 1}$ so that
  $$\left(F^{\ell\tilde{\PP}_j}(u)\right)^{k_{\ell j }} = \sum_{k_{\ell j} = \sum_{a \geq 0} k_{\ell j a}} \frac{k_{\ell j}!}{\prod_{a \geq 0} k_{\ell j a}} \left(\prod_{a \geq 0} \left(\mathsf{N}_{a}^{\ell \tilde{\PP}_j}\right)^{k_{\ell j a}}\right)u^{\sum_{a \geq 0}(2a -1)k_{\ell j a}}.$$
Next use that $$2g = 2g_0 + s(k) + (2a - 1) \sum_{j=1}^n \sum_{\ell \geq 0} \sum_{a \geq 0} k_{\ell j a}$$ to conclude 
\begin{equation*}
    \sum_{g \geq 0 } \mathsf{N}_{g,\vartheta}^{\beta_p} u^{2g} =  \sum_{k \, \vdash \, p} \left(\sum_{g \geq 0} \mathsf{N}_{g,\Delta, \Delta_k}u^{2g + s(k)} \right)\prod_{j=1}^n \prod_{\ell \geq 1} \frac{\ell^{k_{\ell j}}}{k_{\ell j}!} \left(F^{\ell \tilde{\PP}_j}(u)\right)^{k_{\ell j}}. 
\end{equation*}
The result now follows from the fact, see \cite[Lemma 2.20]{bousseau2018thesis}, that
\begin{equation*}\pushQED{\qed}
F^{\ell \tilde{\PP}_j}(u) = \frac{(-1)^{\ell - 1}}{\ell} \frac{1}{2 \sin\left(\frac{\ell u}{2}\right)}.\qedhere \popQED
\end{equation*}
\renewcommand{\qed}{}     
\end{proof}

\begin{proof}[Proof of Theorem A]
  We now combine Corollary~\ref{corollary : scatter to tropical} with Proposition~\ref{Proposition : GW to toric} in order to prove Theorem~\ref{thm:main}. Corollary~\ref{corollary : scatter to tropical} tells us that
  $$ \left[\left\langle\hat{\vartheta}_{r_1},\ldots,\hat{\vartheta}_{r_s}\right\rangle\right]^\mathsf{sym}
    = \sum_p \sum_{k \vdash p} \mathsf{N}^{\Delta,s}_{\mathsf{trop, \Delta_{k}}}(q)\left( \prod_{j=1}^n \prod_{\ell \geq 1} \frac{1}{k_{\ell j}!} \left(\frac{(-1)^{\ell -1}}{\ell} \frac{q^{\frac{1}{2}} - q^{-\frac{1}{2}}}{q^{\frac{\ell}{2}} - q^{-\frac{\ell}{2}}}\right)^{k_{\ell j}}\right)\left(\prod_{j=1}^n t_j^{p_j}\right).$$
    Applying Theorem~\ref{thm : bdyrefined} to $\mathsf{N}^{\Delta,s}_{\mathsf{trop, \Delta_{k}}}(q)$ tells us that, after identifying $q=e^{iu}$, this is equal to
    $$\sum_p \sum_{k \vdash p} \prod_{j=1}^n \prod_{\ell \geq 1} \ell^{k_{\ell j}}\frac{\left(\sum_{g \geq 0} \mathsf{N}_{g,\Delta, \Delta_k}u^{2g + s(k)} \right)}{\left((-i)\left(q^{\frac{1}{2}} - q^{-\frac{1}{2}}\right)\right)^{s(k)}} \left( \prod_{j=1}^n \prod_{\ell \geq 1} \frac{1}{k_{\ell j}!} \left(\frac{(-1)^{\ell -1}}{\ell} \frac{q^{\frac{1}{2}} - q^{-\frac{1}{2}}}{q^{\frac{\ell}{2}} - q^{-\frac{\ell}{2}}}\right)^{k_{\ell j}}\right)\left(\prod_{j=1}^n t_j^{p_j}\right).$$ Since $s(k) = \sum_{j=1}^n \sum_{\ell \geq 1} k_{\ell j}$ and since under the change of variables we have
    $$q^{\frac{\ell }{2}} - q^{-\frac{\ell }{2}} = 2 i \sin\left(\frac{\ell u}{2}\right),$$
    we have that
    $$ \left[\left\langle\hat{\vartheta}_{r_1},\ldots,\hat{\vartheta}_{r_s}\right\rangle\right]^\mathsf{sym} =  \sum_{p} \sum_{k \, \vdash \, p} \left(\sum_{g \geq 0} \mathsf{N}_{g,\Delta, \Delta_k}u^{2g + s(k)} \right)\prod_{j=1}^n \prod_{\ell \geq 1} \frac{\ell^{k_{\ell j}}}{k_{\ell j}!} \left(\frac{(-1)^{\ell - 1}}{\ell} \frac{1}{2 \sin\left(\frac{\ell u}{2}\right)}\right)^{k_{\ell j}}\left(\prod_{j=1}^n t_j^{p_j}\right),$$ and so the theorem follows from Proposition~\ref{Proposition : GW to toric}.
    
\end{proof}

%%%%%%%%%%%%%%%%%%%%%
% References
%%%%%%%%%%%%%%%%%%%%%

\newcommand{\etalchar}[1]{$^{#1}$}
\providecommand{\bysame}{\leavevmode\hbox to3em{\hrulefill}\thinspace}

\end{document}